\documentclass[a4paper,12pt]{amsart}

\title{Towards a Taub-Bolt to Taub-NUT via Ricci flow with surgery}
\author{John Hughes}
\usepackage{mathtools}
\usepackage{amsmath}
\usepackage{amssymb}
\usepackage{amsfonts}
\usepackage{amsthm}
\newtheorem{condition}{Condition}[section]
\newtheorem{lemma}{Lemma}[section]
\newtheorem{theorem}{Theorem}[section]
\newtheorem{corollary}{Corollary}[section]
\newtheorem{proposition}{Proposition}[section]

\theoremstyle{definition}
\newtheorem{definition}[theorem]{Definition}
\newtheorem{remark}[theorem]{Remark}

\usepackage{graphicx}
\graphicspath{ {./images/} }
\usepackage{amsmath, amssymb, amsfonts, graphicx, pgfplots, tikz, tikz-cd, stmaryrd}
\newtheorem*{question*}{Question}
\date{}
\usepackage{enumerate}
\usepackage{hyperref}
\usepackage{fullpage}
\usepackage{tensor}
\usepackage{bigints}
\address{Mathematical Institute, University of Oxford, Oxford OX2 6GG, United Kingdom}
\email{john.hughes@maths.ox.ac.uk}
\begin{document}

	\maketitle 
	\begin{abstract}
		This paper shows for the first time the existence of a Ricci flow with surgery  with local topology change $\mathbb{CP}^2\setminus\{ \mathrm{pt}\} \rightarrow \mathbb{R}^4$. The post surgery flow converges to the Taub-NUT metric on $\mathbb{R}^4$ in infinite time.
	\end{abstract}
	\tableofcontents

	\section{Introduction}
	If $(M,g_{0})$ is a Riemannian manifold, the Ricci flow is an equation for the evolution of $g_{0}$:
	\begin{equation}\label{RFequation}
		\partial_{t}g(t)=-2Ric(g(t)), \text{ } g(0)=g_{0}.
	\end{equation}

	In harmonic coordinates the Ricci tensor for a metric $g$ is
	
	$$R_{ij}=-\frac{1}{2}\Delta_{g} g_{ij}+ \text{ lower order terms.}$$
	
	This suggests that the Ricci flow is a heat equation for a metric on a manifold and so one might hope that, under the flow, the metric would regularise over time and converge to a canonical metric on $M$.  The Ricci flow equation however is non-linear,
	due to both the Laplacian depending on $g$, and the lower order terms, which can
	cause finite time singularities. This can most easily be seen by looking at
	the evolution equation for the scalar curvature under the Ricci flow on an
	$n$-dimensional manifold:
	$$\partial_t R= \Delta R+2|Ric|^2 \geq \Delta R+\frac{2}{n}R^2.$$
	
	The term $\frac{2}{n}R^2$ causes a finite time blow up if $\inf_M R>0$ at time zero. As a result, the regularising process of the flow may encounter a finite time singularity before it has finished. This motivates the idea of Ricci flow with surgery, where one resolves the singularities as they occur along the flow, allowing the regularising process to continue. This was used by Perelman \cite{Per1}, \cite{Per2}, \cite{Per3} in dimension three to prove the Poincar\'{e} conjecture. This paper proves the existence of Ricci flow with surgery in four dimensions where, for the first time, the local and global topology change of the underlying manifold is $\mathbb{CP}^2\setminus\{ \mathrm{pt}\} \rightarrow \mathbb{R}^4$, where 
	$\mathbb{CP}^2\setminus\{ \mathrm{pt}\}$ is diffeomorphic to the vector bundle $\mathcal{O}(-1)$. The curvature of the Ricci flow $G(t)$ of Theorem \ref{fullthm} below will blow up in finite time either only at the zero section of $\mathcal{O}(-1)$ or on a tubular neighbourhood of the zero section, or the closure of such a neighbourhood. A major reason why the Ricci flow of Theorem \ref{fullthm} is interesting is that the regularising process of the Ricci flow does not evolve the initial metric $G(t_0)$, which is equal to a member of the family of $U(2)$-symmetric non-K\"{a}hler Ricci flat Taub-Bolt metrics \cite{Bolt} on $\mathbb{CP}^2\setminus\{\text{pt}\}$ outside of some compact set, to a (perhaps different) member of the Taub-Bolt family but instead flows to a member of the hyperk\"{a}hler Taub-NUT family \cite{NUT} on $\mathbb{R}^4$ by either crushing the zero section or a tubular neighbourhood of the zero section to a point or line respectively.  This paper will be interested in Ricci flow that is invariant under a $U(2)$-action on $\mathcal{O}(-1)$, whose action will be described in \ref{TBandFIK}, with the set of fixed points under this action being the zero section of $\mathcal{O}(-1)$. This set of fixed points will be called the bolt.
	
	Fixed points of the Ricci flow are Ricci flat metrics. It is therefore natural to ask about the stability of such metrics under this flow. This paper is inspired by the work of Holzegel, Schmelzer and Warnick \cite{Cla1} where the  stability of a member $g_{\mathrm{Bolt}}$ of the family of Taub-Bolt metrics on $\mathbb{CP}^2\setminus\{ \mathrm{pt}\}$ \cite{Bolt} is analysed. The family of Taub-Bolt metrics will be introduced in more detail in \ref{TN} and will in particular show that all members are isometric up to rescaling. Thus, studying the stability of one member is equivalent to studying the stability of any other member.  The authors above find numerically that the linearisation $\mathcal{D}_{g_{\mathrm{Bolt}}}(-2\text{Ric}_{g})$ of $-2\text{Ric}_{g}$ about $g_{\mathrm{Bolt}}$ has a positive eigenvalue. They call the corresponding symmetric eigentensor the `unstable' direction. The evolution was simulated on a computer to show that if Taub-Bolt is perturbed in the unstable direction so that the size of the bolt is decreased, then under the Ricci flow, the size of the bolt shrinks to zero in finite time with the curvature only blowing at the bolt, giving a finite time singularity. After surgery, the metric flowed to a member of the Taub-NUT family, which will be introduced in more detail in \ref{TN}. The surgery involves the topology change of $\mathbb{CP}^2\setminus \{\text{pt}\}$ to $\mathbb{R}^4$ mentioned above. It is has been conjectured by many that the singularity is modelled on the shrinking K\"{a}hler Ricci soliton FIK \cite{FIK}, which will be described in more detail in \ref{FIK}.
	
	Actually proving the numerical findings of \cite{Cla1} analytically is challenging.
	This paper makes progress towards proving such a conjecture. 
	\begin{theorem} \label{fullthm}
		There is a $U(2)$-symmetric Ricci flow $G(t)$ with surgery such that
		\begin{enumerate}[i)]
			\item The pre-surgery flow $G(t)$ is defined on $[0,T)  \times \mathbb{CP}^2\setminus\{\text{pt}\}$ and $G(t_0)=g_{\mathrm{Bolt}}$ outside some compact set of $\mathbb{CP}^2\setminus\{\text{pt}\}$.
			\item A finite time singularity is formed at time $T$, with the bolt shrinking to zero size and the curvature blows up exactly on either the bolt or some finite size tubular neighbourhood containing the bolt, or the closure of such a neighbourhood.
			\item The singularity is of Type I and is modelled on the K\"{a}hler shrinking Ricci soliton FIK.
			\item Surgery can performed near the bolt to produce a metric $\tilde{G}$ on $\mathbb{R}^4$, with the local topology change $\mathbb{CP}^2\setminus\{ \mathrm{pt}\} \rightarrow \mathbb{R}^4$.
			\item Then the Ricci flow $G(t)$ starting at $\tilde{G}$ is immortal and converges to a member of the Taub-NUT family in the pointed Cheeger-Gromov sense as $t \rightarrow \infty$.
		\end{enumerate}
	\end{theorem}
	To produce the finite time singuality modelled on FIK, we will use a Wa{\.z}ewski box argument similar to that deployed by Stolarksi \cite{Stol1}, where it is shown that given an asymptotically conical shrinking soliton, there is a Ricci flow on a compact manifold which forms a local finite time singularity modelled on it. The family of metrics used in the Wa{\.z}ewski box argument \cite{Stol1} is invariant under the action of a finite group. The simplest example, which they construct in \cite[Proposition A.7]{Stol1}, is two copies of a family of perturbations of the given soliton glued onto the two ends of a finite cylinder, so as to have a $\mathbb{Z}_2$-symmetry. The Ricci flow is invariant under this symmetry and so the author only has to analyse the Ricci flow on one 'half' of the manifold. Inspired by this, we will use a family of perturbations of the asymptotically conical shrinking soliton FIK, which we denote by $g_{\mathrm{FIK}}$, glued into $g_{\mathrm{Bolt}}$, to give a metric on a non-compact manifold. This glueing construction echoes that used in the authors proof of the $L^2$-instability of Taub-Bolt \cite{hughes2024l2instabilitytaubboltmetricricci}. Since we do not close up the manifold using two copies of this glueing procedure, analysis on the whole of our manifold will be carried out similar to that of the analysis of Stolarski on the one half of their compact manifold.  However, unlike Stolarski, who doesn't have any symmetry assumptions beyond the $\mathbb{Z}_2$-symmetry, our family of perturbations of $g_{\mathrm{FIK}}$ will be $U(2)$-symmetric.  This $U(2)$-symmetric version of the Wa{\.z}ewski box argument of \cite{Stol1} will be used to obtain a Ricci flow within a special class $G_{AF, \mathbb{CP}^2-\{\mathrm{pt}\}}$ of $U(2)$-symmetric metrics that encounters a finite time singularity modelled on FIK. From the numerical work in \cite{Oxford1}, we expect this blow up behaviour for quite general initial conditions. However the non-constructive box argument only shows the existence of a Ricci flow with the properties stated in Theorem \ref{fullthm}. The reason for the use of a non-constructive argument is that even proving that Ricci flow starting at some $g\in G_{AF, \mathbb{CP}^2-\{\mathrm{pt}\}}$  leads to a finite time singularity is not straightforward. By contrast, in many other scenarios, see \cite{App1}, \cite{Dan2} and \cite{Stol2} for example, showing the existence of a finite time singularity follows easily from a simple use of a maximum principle, with the hard the work only  occurring in a more detailed analysis of the singularity formation. Metrics in $G_{AF, \mathbb{CP}^2-\{\mathrm{pt}\}}$ have curvature decay at infinity and so, by the pseudolocality result \cite{pseudolocality}, a global maximum principle cannot be applied to a Ricci flow starting at $g$. This means proving Theorem \ref{fullthm}  will require a more subtle argument. 	 
	As explained in \cite{Stol1}, the Ricci flow $G(t)$ before surgery can be made to be non-K\"{a}hler by perturbing the initial metric outside a compact set. The surgery process will be flexible enough to ensure that the post surgery initial metric will also be non-K\"{a}hler. We point that the property of $G(t_0)$ being equal to $g_{\mathrm{Bolt}}$ outside some compact set, although in spirit of the instability conjecture of Holzegel, Schmelzer and Warnick, is not necessary to prove the existence of a Ricci flow with all the other properties in Theorem \ref{fullthm}. Indeed, the Wa{\.z}ewski box argument is also flexible enough to allow us to glue $g_{\mathrm{FIK}}$ into any metric in $G_{AF, \mathbb{CP}^2-\{\mathrm{pt}\}}$, of which $g_{\mathrm{Bolt}}$ is a member. 
	
	The $U(2)$-symmetry reduces the Ricci flow to a 2 dimensional system of equations, and so significantly simplified the Ricci flow equation. This will allow us to determine, for Ricci flow in $G_{AF, \mathbb{CP}^2-\{\mathrm{pt}\}}$, exactly where it is possible for the curvature to blow up in finite time. We point out that the proof of Theorem \ref{fullthm} will not show that the curvature only blows up at the bolt. It could be the case that the curvature blows up on of some tubular neighbourhood containing the bolt, or the closure of such a neighbourhood.

	The surgery process will be performed in such a way that the post surgery metric lies in the special class of $U(2)$-symmetric metrics $G_{AF, \mathbb{R}^4}$ introduced in \cite{FraNUT} and the dynamical stability result \cite[Theorem 1]{FraNUT} is used to show the convergence of the flow to the Taub-NUT metric. The surgery will preserve the $U(2)$-symmetry and is very explicit. We will show that it can performed arbitrarily close to the singularity time and on an arbitrarily small set containing the points at which the curvature blows up. We note that rotationally invariant Ricci flow with surgery has been studied considerably: the existence of such a flow with surgery resolving both degenerate and non-degenerate neckpinch singularities has been shown  \cite{MR3544617}, \cite{MR2995432}. Recently, rotationally invariant Ricci flow with surgery starting at any closed rotationally
	invariant Riemannian manifold was constructed \cite{MR4806200}. \\

	Outline of the paper:\\
	
	Section \ref{TBandFIK} will discuss $U(2)$-symmetric cohomogeneity one metrics on $\mathbb{CP}^2\setminus\{\text{pt}\}$ and $\mathbb{R}^4$. In particular, we will introduce the Taub-NUT and FIK metrics. Also, a $U(2)$-symmetric spectral decomposition of the weighted Lichnernowicz Laplacian associated with $g_{\mathrm{FIK}}$ will be constructed.  \\
	
	Section \ref{GAFsection} will present a result \cite{FraNUT} which says that, within a class $G_{AF, \mathbb{R}^4}$ of metrics, the family of Taub-NUT metrics are global attractors under the Ricci flow. An analogue $G_{AF, \mathbb{CP}^2\setminus\{\text{pt}\}}$ will be defined and Ricci flow within this class will be analysed. \\
	
	Section \ref{setupsection} will set up the Wa{\.z}ewski box argument, which will only differ from the set up of \cite{Stol1} in that the eigenmodes of the weighted Lichnerowicz Laplacian will be $U(2)$-symmetric and FIK, our asymptotically conical shrinker, will be glued into $\mathbb{CP}^2\setminus\{\text{pt}\}$ with a metric which is equal to $g_{\mathrm{Bolt}}$ outside some compact set. \\
	
	Section \ref{final} will prove Theorem \ref{fullthm2} which is a slightly stronger statement than Theorem \ref{fullthm} that can only be stated after introducing the notation of Section \ref{GAFsection}. Theorem \ref{fullthm2} asserts that, up to diffeomorphism, the member of the Taub-NUT family that the post surgery flow converges to is completely determined by the asymptotics of the pre surgery initial condition $G(t_0)$. Using Section \ref{setupsection}, the existence of a Ricci flow in $G_{AF, \mathbb{CP}^2\setminus\{\text{pt}\}}$ will be proven which encounters a finite time singularity modelled on FIK. Using Section \ref{GAFsection}, the possible regions on which the curvature blows up will be identified which will allow surgery to be performed to then yield a Ricci flow that converges to a member of the Taub-NUT family. \\

	$\mathbf{Acknowledgements.}$ I would like to thank my
	supervisor Jason Lotay for his support and advice.

	$\textbf{Conflicts of interest}.$
	Research supported by a scholarship from EPSRC (grant number EP/W524311/1).

	\section{The Taub-NUT, Taub-Bolt and FIK metrics}\label{TBandFIK}
	In this section we introduce two $U(2)$-symmetric cohomogeneity one metrics,  the Taub-NUT and Taub-Bolt families of Ricci flat metrics, and the K\"{a}hler Ricci shrinker FIK. But first we introduce the concept of a nut and a bolt of a Riemannian manifold with an $S^1$ isometry group and discuss $SU(2)$-symmetric four dimensional cohomogeneity one metrics. After that a $U(2)$-symmetric spectral decomposition of the weighted Lichnernowicz Laplacian associated with $g_{\mathrm{FIK}}$ will be shown to exist.
	\subsection{Nuts and bolts}
	Suppose $(M,g)$ is a four dimensional connected complete Riemannian manifold with an $S^1$ group of isometries. Given a fixed point $p\in M$ of this group of isometries, one of the following holds \cite{Haw1}:

	\begin{enumerate}[i)]
		\item The point $p$ is an isolated fixed point. This fixed point is called a nut.
		\item The point $p$ is part of a two dimensional submanifold of fixed points. This submanifold is called a bolt.
		\item The whole space $M$ is fixed by the action.
	\end{enumerate}

	\subsection{Four dimensional cohomogeneity one $SU(2)$-symmetric metrics}\label{U(2)}
	We follow the notation and set up of \cite{Dancer}. Let $(M^{4},g)$ be a Riemannain manifold and let $SU(2)$ act by isometries on $M$ with $3$ dimensional principal orbits (i.e. $g$ is of cohomogeneity one). Let $\gamma(s)$ be a unit speed geodesic which intersects all principal orbits orthogonally. Suppose the principal isometry group along $\gamma(s)$ is trivial. Then there is an equivariant diffeomorphism 
	\begin{equation*}
		\Phi: I\times (SU(2))\rightarrow M_{0}
	\end{equation*}
	defined by $\Phi(s,g)=g\cdot \gamma(s)$, where $M_{0}$ is a dense subset of $M$ and $I$ is an open interval of $\mathbb{R}$. It follows that 
	\begin{equation*}
		\Phi^{*}(g)=ds^2+g_{s},
	\end{equation*}
	where $g_{s}$ is a one-parameter family of $SU(2)$-invariant metrics on $SU(2)=S^3$. 
	
	We will be interested in the cases where $M=\mathbb{CP}^2\setminus\{\text{pt}\}$ or $M=\mathbb{R}^4$, (we will explain how $SU(2)$ acts on $M$ in $\ref{boundaryconditions}$). Then 
	\begin{equation*}
		\Phi^{*}(g)=ds^2+g_{s}
	\end{equation*}
	for $g_{s}$ a one-parameter family of $SU(2)$-invariant metrics on $SU(2)=S^3$. We now describe the $SU(2)$-invariant metrics on $S^3$.
	We can put coordinates on $S^{3}$ via the map: 
	\begin{equation} \label{S^3 coordinates}
		(\theta, \phi, \psi) \mapsto \left(\cos\left(\frac{\theta}{2}\right)e^{\frac{\psi+\phi}{2}i}, \sin\left(\frac{\theta}{2}\right)e^{\frac{\psi - \phi}{2}i}\right),
	\end{equation}
	where $0<\theta<\pi, 0<\phi<2\pi, 0<\psi<4\pi$. The coordinates come from the Hopf fibration $\pi: S^{3} \rightarrow S^{2}$ where $\theta, \phi$ describe the $S^{2}$ and $\psi$ is a coordinate on the fibres. 
	Define the following basis of $SU(2)$-invariant $1$-forms
	\begin{align*}
		&\sigma_{1}= \frac{1}{2}\left(-\sin\psi d\theta+ \cos\psi \sin\theta d\phi\right),\\
		&\sigma_{2}= \frac{1}{2}\left(\cos\psi d\theta + \sin\psi \sin\theta d\phi\right),\\
		&\sigma_{3}= \frac{1}{2}\left(d\psi+ \cos\theta d\phi\right).
	\end{align*}
	If $X_{1}, X_{2}, X_{3}$ is dual to $\sigma_{1}, \sigma_{2}, \sigma_{3}$, then $X_{3}=\partial_{\psi}$ is tangent to the Hopf fibres.  We also have $\sigma_{1}^2+\sigma_{2}^2= \pi^{*}g_{S^2\left({\frac{1}{2}}\right)}$, where $S^2\left(\frac{1}{2}\right)$ is the sphere of radius $\frac{1}{2}$.
	The ${\sigma_{i}}'s$ are normalised so that the standard metric on $S^{3}$ of radius 1 can be written in the simply form:
	\begin{equation*} 
		g_{S^3_{1}}=\frac{1}{4}\left(\left(d\psi + \cos\theta d\phi\right)^2 + d\theta^2+ \sin^{2}\theta d\phi^2\right)= \sigma_{1}^2 + \sigma_{2}^2 + \sigma_{3}^2.
	\end{equation*}  
	
	Four dimensional cohomogeneity one $SU(2)$-symmetric metrics $g$ defined on $M$ can be written on an open dense subset $I \times S^3$ as:
	\begin{equation*}
		g=ds^2+ a^2(s)\sigma_{1}^2+b^2(s)\sigma_{2}^2+c^2(s)\sigma_{3}^2,
	\end{equation*}
	where $a(s), b(s), c(s)>0$ for all $s\in I$ \cite{Hawking}.

	Viewing $S^{3} \subset \mathbb{C}^2$, then multiplying the coordinates by unit complex numbers gives a $U(1)$-action on $S^{3}$, generated by $\partial_{\psi}$. Demanding that the metric be invariant under this $U(1)$-action on $S^3$ forces $a=b$, and hence results in the following $U(2)$-symmetric metric on $I \times S^3$
	\begin{equation}\label{local}
		\begin{split}
			g&=ds^2+ b^2(s)\left(\sigma_{1}^2+\sigma_{2}^2\right)+c^2(s)\sigma_{3}^2\\
			&=ds^2+ b^2(s)\pi^{*}g_{S^2({\frac{1}{2}})}+c^2(s)\sigma_{3}^2.
		\end{split}
	\end{equation}
	Throughout the rest of this paper the following notation will sometimes be used: that for a function $f$ of a variable $x$, we write $f_x\coloneqq\partial_x f$. 
	\subsection{Boundary conditions} \label{boundaryconditions}
	Let 
	\begin{equation}\label{eq:1}
		g=ds^2+ b^2(s)\left(\sigma_{1}^2+\sigma_{2}^2\right)+c^2(s)\sigma_{3}^2, \qquad s>0,
	\end{equation}
	be a metric on $\mathbb{R}_{>0} \times S^3$, where $b(s), c(s) >0$ for all $s>0$ are smooth functions (i.e. $s$ is the standard coordinate on $\mathbb{R}_{>0}$ and the $\sigma_{i}$'s are 1-forms on $S^3$). If one, or both, of $\lim_{s\rightarrow 0}b(s)=0$ or $\lim_{s\rightarrow 0} c(s)=0$ holds, then the metric (\ref{eq:1}) can close up at $s=0$ to give a $U(2)$-symmetric metric on a topology different to $\mathbb{R}_{>0} \times S^3$. In this subsection we give a condition on $b,c$ so that a metric (\ref{eq:1}) closes up to give a well-defined metric on $\mathbb{R}^4$ or $\mathbb{CP}^2\setminus\{\text{pt}\}$.

	\subsubsection{$\mathbb{R}^4$}
	We can view $\mathbb{R}^4$ as $\mathbb{R} \times S^3$ with a point added at the origin, and $U(2)$ acts in the obvious way.  The conditions for a metric (\ref{eq:1}) to close up at $s=0$ (i.e. at the origin) to give a well-defined metric on $\mathbb{R}^4$ are:
	\begin{condition}\label{R^4conds}
		\hfill\break
		\vspace*{-0.2in}
		\begin{enumerate}[i)]
			\item $b(s)$ and  $c(s)$ can be extended to smooth odd functions on $\mathbb{R}$,
			\item  $b_{s}(s), c_{s}(s) \rightarrow 1$ as $s \rightarrow 0$.
		\end{enumerate}
	\end{condition}
	With $\psi$ defined via (\ref{S^3 coordinates}), the vector field $s\partial_{\psi}$ on $S^3$ can be extended to a vector field on $\mathbb{R}^4$.
	For a metric (\ref{eq:1}) satisfying Condition \ref{R^4conds}, the origin is the nut of the metric since it is the only fixed point of the $U(1)$ group of isometries generated by $s\partial_{\psi}$.

	\subsubsection{$\mathbb{CP}^2\setminus\{\mathrm{pt}\}$}
	
	The space $\mathbb{CP}^2\setminus\{\text{pt}\}$ is diffeomorphic to the vector bundle $\mathcal{O}(-1)$ over $S^2$, which is $\mathbb{C}^2$ blown up at the origin. Thus $\mathbb{CP}^2\setminus\{\text{pt}\}$ can be thought of as $\mathbb{C}^2\setminus\{\text{pt}\}= \mathbb{R} \times S^3$ with a $\mathbb{CP}^1$ glued at the origin, and $U(2)$ acts in the obvious way on $\mathbb{R} \times S^3$ while fixing the $\mathbb{CP}^1$ at the origin. The conditions for a metric (\ref{eq:1}) to close up at $s=0$ (i.e. at the $\mathbb{CP}^1$ located at the origin) to give a well defined metric on $\mathbb{CP}^2\setminus\{\text{pt}\}$ are:
	\begin{condition}\label{0(-1)cond}
		\hfill\break
		\vspace*{-0.2in}
		\begin{enumerate}[i)] 
			\item $b(s)$ and $c(s)$ can be extended to even and odd smooth functions respectively on $\mathbb{R}$,
			\item $b(s) \rightarrow \beta >0$ as $s \rightarrow 0$.
		\end{enumerate}
	\end{condition}
	With $\psi$ defined via (\ref{S^3 coordinates}), the vector field $s\partial_{\psi}$ on $S^3$ can be extended to a vector field on $\mathbb{CP}^2\setminus\{\text{pt}\}$.  
	For a metric (\ref{eq:1}) satisfying Condition \ref{0(-1)cond}, the 2-sphere at $s=0$ is the bolt of this metric (\ref{eq:1}) since it is the fixed points set of the $U(1)$ group of isometries generated by $s\partial_{\psi}$. The bolt has radius $b(0)$, and so has area/size $4\pi b^2(0)$.
	
	\subsection{Three families of four dimensional cohomogeneity one $U(2)$-symmetric metrics}

	\subsubsection{Taub-NUT}\label{TN}
	The family of Taub-NUT metrics on $\mathbb{R}^4$ was first discovered by Newman, Unti, Tamburino \cite{NUT}, and consists of all the metrics $\alpha g_{\mathrm{NUT},n}$ for $\alpha,n>0$, where $g_{\mathrm{NUT},n}$ can be written in a global frame as:
	\begin{equation}\label{TNeq}
		g_{\mathrm{NUT},n}=\frac{r+2n}{r}{dr}^{2}+ 4r(r+2n)\left(\sigma_{1}^{2}+\sigma_{2}^{2}\right)+ 16n^{2}\frac{r}{r+2n}\sigma_{3}^2, \text{ }r>0.
	\end{equation}
	The coordinate $r$ is a coordinate on $\mathbb{R}_{>0}$, and the ${\sigma_{i}}$'s are as in \ref{U(2)}. This makes the local expression (\ref{TNeq}) a metric on $\mathbb{R}_{>0} \times S^3$. When $$s= \int_{0}^{r} \sqrt{\frac{\bar{r}+2n}{\bar{r}}}d\bar{r},$$
	equation (\ref{TNeq}) can be written in the form (\ref{eq:1}) with $b,c$ satisfying the Condition \ref{R^4conds} to give a metric on $\mathbb{R}^4$. With suitable complex structures on $\mathbb{R}^4$ each Taub-NUT metric $\alpha g_{\mathrm{NUT,n}}$ is hyperk\"{a}hler. Up to rescaling and pullback by diffeomorphisms, the Taub-NUT family has only one member. Indeed, if we rescale  $\phi(r)=\frac{n}{m}r$  then $\phi^{*}g_{\mathrm{NUT},n}=\frac{n^2}{m^2}g_{\mathrm{NUT},m}$.
	
	\subsubsection{Taub-Bolt}
	The family of Taub-Bolt metrics on $\mathbb{CP}^2\setminus\{\text{pt}\}$  was first discovered by Page \cite{Bolt}, and consists of all the metrics $\alpha g_{\mathrm{Bolt},n}$ for $\alpha, n>0$, where $g_{\mathrm{Bolt},n}$ can be written in a local frame as:
	\begin{equation}\label{TB}
		g_{\mathrm{Bolt},n}=\frac{(r+n)(r+3n)}{r\left(r+\frac{3n}{2}\right)}dr^{2}+ 4(r+n)(r+3n)\left(\sigma_{1}^{2}+\sigma_{2}^{2}\right)+ 16n^{2}\frac{r\left(r+\frac{3n}{2}\right)}{(r+n)(r+3n)}\sigma_{3}^2,
	\end{equation}
	for $r>0$. The coordinate $r$ is a coordinate on $\mathbb{R}_{>0}$, and the ${\sigma_{i}}'s$ are as in Subsection \ref{U(2)}. This makes the local expression (\ref{TB}) a metric on $\mathbb{R}_{>0} \times S^3$.
	When $$s= \int_{0}^{r} \sqrt{\frac{(\bar{r}+n)(\bar{r}+3n)}{\bar{r}\left(\bar{r}+\frac{3n}{2}\right)}} d\bar{r},$$ equation (\ref{TB}) can be written in the form (\ref{eq:1}) with $b,c$ satisfying Condition \ref{0(-1)cond} to give a metric on $\mathbb{CP}^2\setminus\{\text{pt}\}$.
	
	We note that the metrics $g_{\mathrm{Bolt},n}$ and $g_{\mathrm{Bolt},m}$ are isometric up to scaling. Indeed, if  we consider the diffeomorphism $\phi(r)=\frac{n}{m}r$ of $\mathbb{CP}^2-\{\text{pt}\}$ then $\phi^{*}g_{\mathrm{Bolt},n}=\frac{n^2}{m^2}g_{\mathrm{Bolt},m}$. For the rest of the paper, we will set $g_{\mathrm{Bolt}}=g_{\mathrm{Bolt},1}$.
	
	\subsubsection{FIK}\label{FIK}
	A Ricci soliton $(M,g,X)$ is a triple consisting of a Riemannian manifold $(M,g)$ and a vector field $X$ on $M$ such that 
	\begin{equation}\label{solitoneq}
		Ric+ \mathcal{L}_{X}g=\lambda g,
	\end{equation}
	for some $\lambda \in \mathbb{R}$. Let $\phi_{t}:M\rightarrow M$ be the family of diffeomorphisms satisfying 
	\begin{equation}\label{phi_t}
		\partial_{t} \phi_{t}=\frac{1}{1-2\lambda t} X \circ \phi_{t}, \qquad \phi_{0}=\mathrm{Id}_{M}.
	\end{equation}
	Then a solution to the Ricci flow equation (\ref{RFequation}) is given by $$g(t)=(1-2\lambda t)\phi_{t}^* g.$$ If $X=\nabla f$ for some function $f$ then $(M,g,X)=(M,g,f)$ is called a gradient Ricci soliton. 
	An example of a Ricci soliton central to this paper is $(\mathbb{CP}^2\setminus\{\text{pt}\}, g_{\mathrm{FIK}}, f_{\mathrm{FIK}})$, a K\"{a}hler shrinking ($\lambda>0$) gradient Ricci soliton called FIK.
	\begin{theorem}\label{FIKthm}
		Let $C_{\mathrm{FIK}}\coloneqq (C(S^3), g_{C_{\mathrm{FIK}}})$ be the cone over the sphere $(S^3, g_{S^3_\mathrm{FIK}} \coloneqq \frac{1}{\sqrt{2}}(\sigma_{1}^2+\sigma_{2}^2)+\frac{1}{2}\sigma_{3}^2)$. Then there exists a $U(2)$-symmetric gradient K\"{a}hler Ricci shrinker $(\mathbb{CP}^2\setminus \{\mathrm{pt}\}, g_{\mathrm{FIK}}, f_{\mathrm{FIK}})$ (with $\lambda=\frac{1}{2}$ in (\ref{solitoneq})) that is asymptotically conical with asymptotic cone $C_{\mathrm{FIK}}$. 
		When $g_{\mathrm{FIK}}$ is written in the form 
		\begin{equation*}
			g_{\mathrm{FIK}}=ds^2+ b_{\mathrm{FIK}}^2(s)\left(\sigma_{1}^2+\sigma_{2}^2\right)+c_{\mathrm{FIK}}^2(s)\sigma_{3}^2, \qquad s>0,
		\end{equation*}
		we have the following:
		\begin{enumerate}[i)]
			\item $b(s) \rightarrow \frac{s}{\sqrt[4]{2}}$, $c(s) \rightarrow \frac{s}{\sqrt{2}}$ as $s\rightarrow \infty$;
			\item There exists $\delta>0$ such that $\frac{c}{b}(s)\leq 1-\delta$, $b_s,c_s>0$ for all $s>0$:
			\item $b_s=\frac{c}{b}$ for all $s>0$.
		\end{enumerate}
		Finally, let $\phi_{t}$ be the family of diffeomorphisms satisfying (\ref{phi_t}) with $X=\nabla f_{\mathrm{FIK}}$, and let $g_{\mathrm{FIK}}(t)=(1-t)\phi^*_t g_{\mathrm{FIK}}(0)$ denote the Ricci flow with initial condition $g_{\mathrm{FIK}}$. Then there exists a smooth family of diffeomorphisms $$\Psi_t =  \phi_t \circ \Psi : C_{\mathrm{FIK}}  \to (\mathbb{CP}^2\setminus \{\mathrm{pt}\})\setminus \{s=0\}$$ such that 
		$$ ( 1 - t) \Psi_t^*  g_{\mathrm{FIK}} \xrightarrow[t \nearrow 1]{ C^\infty_{loc} ( C(S^3), g_{C_{\mathrm{FIK}}} ) }		g_{C_{\mathrm{FIK}}}.$$
	\end{theorem}
	
	\begin{proof}
		Apart from the inequality $\frac{c_{\mathrm{FIK}}}{b_{\mathrm{FIK}}} \leq 1-\delta$ part ii), and part iii), we have that Theorem \ref{FIKthm} is an immediate consequence of \cite{FIK}. We have $\frac{c_{\mathrm{FIK}}}{b_{\mathrm{FIK}}}(0)=0$, and $\frac{c_{\mathrm{FIK}}}{b_{\mathrm{FIK}}} \rightarrow \frac{1}{\sqrt[4]{2}}$ as $s\rightarrow \infty$ using L'H\^{o}pital's rule. By \cite{FIK}, $\frac{c_{\mathrm{FIK}}}{b_{\mathrm{FIK}}}(s) \leq 1$ for all $s\geq 0$. Therefore there exists $\delta>0$ such that $\frac{c_{\mathrm{FIK}}}{b_{\mathrm{FIK}}}\leq 1-\delta$ for all $s\geq0$. 
		
		Part iii) holds because it is the condition for a metric of the form (\ref{eq:1}) on $\mathbb{CP}^2\setminus \{\mathrm{pt}\}$ to be K\"{a}hler with respect to the standard complex structure on $\mathbb{CP}^2\setminus \{\mathrm{pt}\}$, which is the case for FIK \cite{FIK}.
	\end{proof}
	
	To keep the set up of our box argument as similar to \cite{Stol1} to make comparison between the two as easy as possible we choose $\lambda=\frac{1}{2}$ in Theorem \ref{FIKthm}. 	
	For the rest of this paper, we will fix the notation $$(\mathbb{CP}^2\setminus \{\mathrm{pt}\}, \overline{g}, f)\coloneqq ( \mathbb{CP}^2\setminus\{\text{pt}\}, g_{\mathrm{FIK}}, f_{\mathrm{FIK}})$$ with $\phi_{t}$ the associated family of diffeomorphisms. Furthermore, the Levi-Civita connection associated with $\overline{g}=g_{\mathrm{FIK}}$ will be denoted by $\overline{\nabla}$.

	\subsection{$\Delta_{\overline{g},f}+2\overline{Rm}$ and its spectrum}
	Only for the rest of this Section, to ease the notation, \ref{TBandFIK}, set $M=\mathbb{CP}^2\setminus \{\mathrm{pt}\}$.
	Let $\Gamma(M,\text{Sym}^{2}T^{*}M)$ denote the space of sections of $\text{Sym}^{2}T^{*}M$.
	Following \cite{Stol1}, to obtain the family of initial data for Ricci flow, crucial to the application of the Wa{\.z}ewski box argument, we will need a spectral decomposition result for the operator $\Delta_f+2\overline{Rm}$ defined by
	$$h_{ij}\mapsto \Delta_{\overline{g}} h_{ij}-\overline{\nabla}_{\overline{\nabla} f}h_{ij} +2\overline{g}^{kp}\overline{g}^{lq}\overline{R}_{iklj}h_{pq},$$  but now acting on $U(2)$-symmetric tensors in $\Gamma(M,\text{Sym}^{2}T^{*}M)$.

	\begin{definition}
		For any $m \in \mathbb{N}$, $H^m_f$ is the completion of 
		\begin{equation*}
			\left\{ T \in \Gamma(M,\text{Sym}^{2}T^{*}M)  : \sum_{j=0}^m \int_\Omega \lvert \overline{\nabla}^j T\rvert^2_{\overline{g}} e^{-f} dV_{\overline{g}} < \infty  \right\}
		\end{equation*}	
		with respect to the inner product
		$$( S , T )_{H^m_f} \coloneqq  \sum_{j=0}^m \int_\Omega \langle \overline{\nabla}^j S,  \overline{\nabla}^j T \rangle e^{-f} dV_{\overline{g}} .$$
		We write $H^0_{f}$ as $L^2_{f}$.
	\end{definition}

	Define 
	$$D(\Delta_{\overline{g},f})\coloneqq \{h\in H_{f}^1: \Delta_{\overline{g},f}h \in L_f^2\},$$
	where here $\Delta_{\overline{g},f}h$ is taken in the sense of distributions. As explained in \cite{Stol1}, using \cite[Theorem 4.6]{extend} we can extend $\Delta_{\overline{g},f}$ to a self-adjoint operator 
	$$\Delta_{\overline{g},f}: D(\Delta_{\overline{g},f}) \rightarrow L_f^2 (M, \text{Sym}^2 T^*M)$$
	that agrees with the usual $\Delta_{\overline{g},f}$ operator on smooth, compactly supported 2-tensors. 	
	Since $\overline{g}$ is a complete, connected shrinker with bounded curvature, it follows that $\Delta_{\overline{g},f} +2\overline{Rm}$ also defines a self-adjoint operator 
	$$\Delta_{\overline{g},f}+2\overline{Rm}: D(\Delta_{\overline{g},f}) \rightarrow L_f^2 (M, \text{Sym}^2 T^*M),$$
	called the weighted Lichnernowicz Laplacian associated with $\overline{g}$, that is bounded above \cite[Lemma 2.17]{Stol1}  and has compact resolvent \cite[Lemma 2.18]{Stol1}.
	This yields the following theorem about the spectrum of $\Delta_{\overline{g},f} + 2\overline{Rm}$.
	\begin{theorem}[\cite{Stol1}, Theorem 2.19]\label{Spectrum}
		There exists a smooth orthonormal basis $\{ h_j \}_{j = 1}^\infty$ of $L^2_f$ 
		such that, for all $j \in \mathbb{N}$, $h_j \in D(\Delta_{\overline{g},f})$ is an eigenmode of $\Delta_{\overline{g},f} + 2\overline{Rm}$ with eigenvalue $\lambda_j \in \mathbb{R}$,
		and the eigenvalues $\{ \lambda_j \}_{j= 1}^\infty$ satisfy $\lambda_1 \ge \lambda_2 \ge \dots$.
		
		Moreover, the eigenvalues $\{ \lambda_j \}_{j = 1}^\infty$ are each of finite-multiplicity, are given by the min-max principle, and tend to $-\infty$ as $j \to \infty$.
		The spectrum $\sigma ( \Delta_{\overline{g},f} + 2 \overline{Rm} )$ of $\Delta_{\overline{g},f} + 2 \overline{Rm} : D( \Delta_{\overline{g},f}) \to L^2_f$ equals $\{\lambda_j \}_{j = 1}^\infty$.
	\end{theorem}

	\subsection{$U(2)$-symmetric basis of $L^{2,U(2)}_f$}
	Since $\overline{g}$ is symmetric, the following proposition holds.
	\begin{proposition}
		Recall that $U(2)$ acts on $(M, \overline{g})$ by isometries. Then for every $\alpha\in U(2)$ and $h\in \Gamma(M,\text{Sym}^{2}T^{*}M)$ we have
		$$l_\alpha^*((\Delta_{g,f} + 2\overline{Rm})(h))= (\Delta_{g,f} + 2\overline{Rm})(l_\alpha^*h).$$
	\end{proposition}
	
	Let $h \in \Gamma(M,\text{Sym}^{2}T^{*}M)$. Put a $U(2)$-symmetric metric on $U(2)$, and let $dVol$ be the corresponding volume form.
	Define
	$$h_{avg}=\frac{1}{Vol(U(2))}\int_{U(2)} l_\alpha^*(h)dVol.$$
	Then $h$ is $U(2)$-symmetric if and only if $h=h_{avg}$. 
	Now take $h \in \Gamma(M,\text{Sym}^{2}T^{*}M)$ such that $(\Delta_{\overline{g}, f} + 2\overline{Rm})(h)=\lambda h$ for $\lambda \in \mathbb{R}$. Then 
	$$(\Delta_{\overline{g},f} + 2\overline{Rm})(h_{avg})= \lambda h_{avg}.$$
	Consider the set $\{ (h_j)_{avg} \}_{j = 1}^\infty \subset L^2_f(M, Sym^2T^*M)$ where the $h_j$ are the eigentensors of Theorem \ref{Spectrum}.
	Suppose $h \in L^2_f(M, Sym^2T^*M)$ be $U(2)$-symmetric. By Theorem \ref{Spectrum}, $h=\sum_{i=1}^{\infty} \beta_i h_i$ for some $\beta_i\in \mathbb{R}$. Thus,
	\begin{align*}
		h&=h_{avg}\\
		&= \frac{1}{Vol(U(2))}\int_{U(2)} l_\alpha^*(h)d\alpha \\
		&= \frac{1}{Vol(U(2))}\int_{U(2)} l_\alpha^*(\sum_{i=1}^{\infty} \beta_i h_i)d\alpha \\
		& = \sum_{i=1}^{\infty} \frac{\beta_i}{Vol(U(2))}\int_{U(2)} l_\alpha^*(h_i)d\alpha \\
		&= \sum_{i=1}^{\infty} \beta_i (h_i)_{avg}.
	\end{align*} 
	
	Therefore $\{ (h_j)_{avg} \}_{j = 1}^\infty \subset L^2_f(M, Sym^2T^*M)$ spans the space $$L^2_f(M, Sym^2T^*M, U(2))\eqqcolon L^{2,U(2)}_f,$$ of $U(2)$-symmetric tensors in $L^2_f(M, Sym^2T^*M)$. We will also denote by $H^{1,G}_f$ the space of $U(2)$-symmetric tensors in $H^1_f$. After taking a subsequence and finding an orthogonal basis on each eigenspace, we can assume $\{ (h_j)_{avg} \}_{j = 1}^\infty$ forms an orthonormal basis of $L^{2,U(2)}_f$. We have proven the following.

	\begin{theorem}\label{U(2)Spectrum}
		There exists a smooth orthonormal basis $\{ h_j \}_{j = 1}^\infty$ of $L^{2, U(2)}_f$
		such that, for all $j \in \mathbb{N}$, $h_j \in D(\Delta_{\overline{g},f}) \cap L^{2, U(2)}_f$ is an eigenmode of $\Delta_{\overline{g},f} + 2\overline{Rm}$ with eigenvalue $\lambda_j \in \mathbb{R}$,
		and the eigenvalues $\{ \lambda_j \}_{j= 1}^\infty$ satisfy $\lambda_1 \ge \lambda_2 \ge \dots$.
		
		Moreover, the eigenvalues $\{ \lambda_j \}_{j = 1}^\infty$ are each of finite-multiplicity, are given by the min-max principle, and tend to $-\infty$ as $j \to \infty$.
		The spectrum $\sigma ( \Delta_{\overline{g},f} + 2 \overline{Rm} )$ of $$\Delta_{\overline{g},f} + 2 \overline{Rm} : D( \Delta_{\overline{g},f}) \cap L^{2, U(2)}_f \to L^{2, U(2)}_f$$ equals $\{\lambda_j \}_{j = 1}^\infty$.
	\end{theorem}

	\begin{remark}\label{>0}
		By the results \cite[Proposition 2.20, Corollary 2.21]{Stol1} for general shrinkers,
		\begin{equation*}
			\Delta_{\overline{g},f}(Ric(g))+2Rm_{\overline{g}}(Ric(\overline{g}))=Ric(\overline{g}). 
		\end{equation*}
		and
		\begin{equation*}
			\Delta_{\overline{g},f}(\overline{\nabla}^2 f)+2Rm_{\overline{g}}(\overline{\nabla}^{2}f)=0. 
		\end{equation*}
		Since $0 \not= Ric(\overline{g}), \overline{\nabla}_{\overline{g}}^2 f \in H^{1, U(2)}_f$, it follows that $\Delta_{\overline{g},f} + 2\overline{Rm}$ has at least two non-negative eigenvalues.
	\end{remark}

	\section{Ricci flow in $G_{AF, \mathbb{R}^4}$, $G_{AF,\mathbb{CP}^2\setminus\{\mathrm{pt}\}}$}\label{GAFsection}
	As explained in the introduction, in Section \ref{setupsection} a $U(2)$-symmetric version of the Wa{\.z}ewski box argument of \cite{Stol1} will be used to obtain a metric within a special class $G_{AF, \mathbb{CP}^2-\{\mathrm{pt}\}}$ of $U(2)$-symmetric metrics that encounters a finite time singularity modelled on FIK.

	From Subsection \ref{boundaryconditions}, $U(2)$-invariant metrics on $\mathbb{R}^4$ or $\mathbb{CP}^2\setminus\{\text{pt}\}$ take the form
	\begin{equation}\label{a}
		g=ds^2+ b^2(s)\left(\sigma_{1}^2+\sigma_{2}^2\right)+c^2(s)\sigma_{3}^2, \qquad s>0,
	\end{equation}
	with $b,c$ satisfying Conditions \ref{R^4conds} or Conditions \ref{0(-1)cond} respectively. For the rest of the paper, given a metric $g$ of the form (\ref{a}), we define $$u\coloneqq\frac{c}{b}.$$
	Recall the dual frame $X_1, X_2, X_3$ to $\sigma_{1}, \sigma_{2}, \sigma_{3}$ introduced in Subsection \ref{U(2)}. Then the non-zero components of the curvature endomorphism of the metric $g$ with respect to the frame $\partial_{s}=X_0, X_1, X_2, X_3$, are
	\begin{equation} \label{sectcurv}
		\begin{aligned}
			R_{1221}&=b^2\left(4-3u^2-b_{s}^2\right),\\
			R_{1331}&=k_{2332}=u^2\left(u^2-b_{s}c_{s}u^{-1}\right),\\
			R_{0110}&=R_{0220}=-bb_{ss},\\
			R_{0330}&=-cc_{ss},\\
			R_{0123}&=R_{0231}=-\frac{1}{2}R_{0312}=\frac{c_{s}}{c}-\frac{b_{s}}{b}.
		\end{aligned}
	\end{equation}
	By the diffeomorphism invariance of the Ricci flow, if the initial metric is $U(2)$-invariant then it remains so under the Ricci flow. Also, it is easy to see from (\ref{sectcurv}), that the Ricci tensor is diagonal for a metric of the form (\ref{a}). Hence, if we have a Ricci flow $(\mathbb{CP}^2\setminus\{\text{pt}\}, g(t))_{[0,T)}$ with $g(0)$ possessing a $U(2)$-symmetry,
	\begin{equation*}
		g(0)=a^2(r)dr^2+ b^2(r)\left(\sigma_{1}^2+\sigma_{2}^2\right)+c^2(r)\sigma_{3}^2, \qquad r>0,
	\end{equation*}
	then $g(t)$ evolves under the Ricci flow as
	\begin{equation*}
		g(t)=a^2(r,t)dr^2+ b^2(r,t)\left(\sigma_{1}^2+\sigma_{2}^2\right)+c^2(r,t)\sigma_{3}^2, \qquad r>0.
	\end{equation*}
	for some functions $a(r,t),b(r,t),c(r,t)$ now dependent on $t \in [0,T)$.
	Then after re-parametrising using the coordinate $s(r,t)=\int_{0}^{r} a(\bar{r},t) d\bar{r}$, 
	\begin{equation} \label{rf}
		g(t)=ds^2+ b^2(s,t)\left(\sigma_{1}^2+\sigma_{2}^2\right)+c^2(s,t)\sigma_{3}^2, \qquad s>0.
	\end{equation}
	with $b,c$ satisfying Conditions \ref{0(-1)cond}. 
	Under the Ricci flow, the functions $b,c$ in (\ref{rf}) satisfy
	\begin{equation}
		\begin{split}\label{riccifloweqaforbc}
			b_{t}&= \partial_{t} b\rvert_{r}=b_{ss}+\left(\frac{c_{s}}{c}+\frac{b_{s}}{b}\right)b_{s}+2\frac{c^2}{b^3}-\frac{4}{b},\\
			c_{t}&=\partial_{t} c\rvert_{r}=c_{ss}+2\frac{b_{s}c_{s}}{b}-2\frac{c^3}{b^4},
		\end{split}
	\end{equation}
	Throughout the rest of Section \ref{GAFsection}, we will always use the notation $b_{t}= \partial_{t} b\rvert_{r}$ and $c_{t}= \partial_{t} c\rvert_{r}$ for $r$ a fixed in time radial coordinate. 
	
	Definition \ref{G} will say that $G_{AF, \mathbb{CP}^2\setminus\{\mathrm{pt}\}}$ are all those metrics $g$ of the form (\ref{a}) on $\mathbb{CP}^2\setminus \{\mathrm{pt}\}$ which satisfy 
	\begin{enumerate}[i)]
		\item $u \leq 1$,
		\item $b_{s},c_{s}\geq 0$,
		\item $\sup_{p\in \mathbb{CP}^2\setminus\{ \mathrm{pt}\}} s^{2+\epsilon} \left\lvert \mathrm{Rm}_{g} \right\rvert _{g}(p(s)) < \infty, \text{ for some } \epsilon>0$.
	\end{enumerate}
	This definition is motivated by the work of Di Giovanni \cite{FraNUT}, where it is shown that if the Ricci flow $g(t)$ is started at a metric $g$ of the form ($\ref{a})$ on $\mathbb{R}^4$ such that 
	\begin{enumerate}[i)]
		\item $u \leq 1$,
		\item $b_{s},c_{s}\geq 0$,
		\item $\sup_{p\in \mathbb{R}^4} (d_{g}(0,p))^{2+\epsilon} \left\lvert \mathrm{Rm}_{g} \right\rvert _{g}(p) < \infty, \text{ for some } \epsilon>0$,
	\end{enumerate}
	then $g(t)$ converges to a member of the Taub-NUT family in the pointed Cheeger-Gromov sense as $t \rightarrow \infty$. Definition \ref{GAF} will say that the class of metrics satisfying these conditions will be called $G_{AF, \mathbb{R}^4}$. 
	The main result of this section will be 
	\begin{theorem}\label{curvblowup1}
		Let $(\mathbb{CP}^2\setminus\{ \mathrm{pt}\},g(t))$ for $[0,T)$ be a complete Ricci flow with $g(0) \in G_{AF, \mathbb{CP}^2\setminus\{ \mathrm{pt}\}}$ with a finite time singularity at $t=T$ and the curvature blowing up at the bolt. Then there exists an $0\leq R<\infty$ (wrt to $g(0)$) 
		$$\{p\in \mathbb{CP}^2\setminus\{ \mathrm{pt}\} : \sup_{[0,T)}\left\lvert Rm_{g(t)} \right\rvert_{g(t)}(p)=\infty\}= \{p\in \mathbb{CP}^2\setminus\{ \mathrm{pt}\} : d_{g(0)}(bolt,p)\in I_R\},$$
		where $I_R$ is the either the interval $[0,R]$ or $[0,R)$.
	\end{theorem}
	
	Theorem \ref{curvblowup1} will be a consequence of the monotonicity of $c(s)$ and the following result which we will prove.
	\begin{proposition}
		Let $(\mathbb{CP}^2\setminus\{\mathrm{pt}\},g(t))$ for $[0,T]$ be a complete Ricci flow with $g(0) \in G_{AF, \mathbb{CP}^2-\{ \mathrm{pt}\}}$. Then there exists $C>0$ such that 
		\begin{equation} \label{cest}
			c^2(s)\left\lvert \mathrm{Rm}_{g(t)} \right\rvert_{g(t)} (s) \leq C.
		\end{equation}
		all $t \in [0,T]$ and $s \geq 0$.
	\end{proposition}
	
	This section will finish with the statement of Theorem \ref{fullthm2}, to be proved in Section \ref{final}, which is a stronger result than Theorem \ref{fullthm} which says that there is Ricci flow $G(t)$ on $\mathbb{CP}^2-\{ \mathrm{pt}\}$ such that $g(t)\in G_{AF, \mathbb{CP}^2\setminus\{\mathrm{pt}\}}$, it encounters a finite time singularity modelled on FIK. Then surgery can be performed to yield a metric on $\mathbb{R}^4$ of the form (\ref{a}) with the properties above. Then the Ricci flow evolves this metric to a member of the Taub-NUT member. Moreover, the member is determined up to action of diffeomorphisms. Indeed, we will see that metrics in $G_{AF, \mathbb{CP}^2\setminus\{\mathrm{pt}\}}$ and $G_{AF, \mathbb{R}^4}$ can be categorised as follows.
	\begin{lemma}
		Let $g \in G_{AF, \mathbb{R}^4}$ or $g\in G_{AF, \mathbb{CP}^2\setminus\{\mathrm{pt}\}}$. Then one of the following holds:
		\begin{enumerate}[i)]
			\item $b_{s} \rightarrow 2, c_{s} \rightarrow 0, c \rightarrow m_{g}^{-1} \in (0, \infty)$ as $s\rightarrow \infty$,
			\item $b_{s} \rightarrow 1, c_{s} \rightarrow 1$ as $s\rightarrow \infty.$
		\end{enumerate}
	\end{lemma}

	We call the $m_g$ the mass of $g$, and it preserved under the Ricci flow. This implies that, since the surgery process will be performed inside a compact of $G_{AF, \mathbb{CP}^2\setminus\{\mathrm{pt}\}}$, the mass of the Taub-NUT metric that the Ricci flow with surgery of Theorem \ref{fullthm} converges to will be equal to the mass of the initial metric of the flow.

	\subsection{Stability of Taub-NUT}\label{future}
	In this section the stability of the Taub-NUT family of metrics within a certain class $G_{AF, \mathbb{R}^4}$ of metrics is discussed. 
	If we write members of the Taub-NUT family of \ref{TN} as 
	\begin{equation} 
		g_{\mathrm{NUT},n}=ds^2+ b^2(s)(\sigma_{1}^2+\sigma_{2}^2)+c^2(s)\sigma_{3}^2, \qquad s\geq0.
	\end{equation}
	for some smooth functions $b(s),c(s)$ depending on $n$ and satisfying Condition \ref{R^4conds}, then $b_{s}=2-u$ and $c_{s}=u^2$, as stated in \cite{FraNUT}. Also $\rvert Rm \lvert$ falls off like $\frac{1}{r^3}$, $b_{s}>0, c_{s}>0$ and $u\leq 1$. Thus, all members of the Taub-NUT family fall into the following class of metrics on $\mathbb{R}^4$, where $0$ denotes the origin in $\mathbb{R}^4$.
	\begin{definition}[\cite{FraNUT} Definition 2.3] \label{GAF}
		Define $G_{AF, \mathbb{R}^4}$ to be the class of all complete metrics $g$ of the form (\ref{eq:1}) on $\mathbb{R}^4$ such that
		\begin{enumerate}[i)]
			\item $u \leq 1$,
			\item $b_{s},c_{s}\geq 0$,
			\item $\sup_{p\in \mathbb{R}^4} (d_{g}(0,p))^{2+\epsilon} \left\lvert Rm_{g} \right\rvert _{g}(p) < \infty, \text{ for some } \epsilon>0$.
		\end{enumerate}
	\end{definition}

	It turns out that metrics in $G_{AF, \mathbb{R}^4}$ split into two categories:
	
	\begin{lemma}[\cite{FraNUT} Lemma 2.4]\label{masswelldfiendforFra}
		Let $g \in G_{AF, \mathbb{R}^4}$. Then one of the following holds:
		\begin{enumerate}[i)]
			\item $b_{s} \rightarrow 2, c_{s} \rightarrow 0, c \rightarrow m_{g}^{-1} \in (0, \infty)$ as $s\rightarrow \infty$,
			\item $b_{s} \rightarrow 1, c_{s} \rightarrow 1$ as $s\rightarrow \infty.$
		\end{enumerate}
	\end{lemma}

	\begin{definition}
		For a metric $g\in G_{AF, \mathbb{R}^4}$, we define the mass of $g$ to be $m_{g}=\lim_{s \rightarrow \infty}\frac{1}{c(s)}$, where in the case of ii) of Lemma \ref{masswelldfiendforFra}, $m_g\coloneqq0$.
	\end{definition}
	
	\begin{lemma}[\cite{FraNUT}] \label{GR4preserved}
		Let $(\mathbb{R}^4, g(t))_{0\leq t<T}$ be a solution to the  Ricci flow with $g(0)\in G_{AF, \mathbb{R}^4}$. Then $g(t) \in G_{AF, \mathbb{R}^4}$ for all $t\in[0,T)$.
	\end{lemma}
	Mass conservation is also shown:
	\begin{corollary}[\cite{FraNUT} Corollary 3.2]
		Let $(\mathbb{R}^4, g(t))_{0\leq t <T\leq \infty}$ be the maximal solution to the Ricci flow starting at some $g_{0}\in G_{AF, \mathbb{R}^4}$ with positive mass $m_{g_{0}}$. Then $m_{g(t)}=m_{g_{0}}$ for all $t\in[0,T)$.
	\end{corollary}
	The Taub-NUT family is shown to be stable in the following sense:
	
	\begin{theorem}[\cite{FraNUT}, Theorem 1] \label{Fra1}
		Let $(\mathbb{R}^4, g(t))_{t \geq 0}$ be the maximal solution to the Ricci flow starting at some $g_{0}\in G_{AF, \mathbb{R}^4}$ with positive mass $m_{g_{0}}$. Then $g(t)$ is an immortal solution which converges to a member of the Taub-NUT family with mass $m_{g_0}$ in the pointed Cheeger-Gromov sense as $t \rightarrow \infty$.
	\end{theorem}

	\subsection{Ricci flow in $G_{AF,\mathbb{CP}^2\setminus\{\text{pt}\}}$}\label{GAFpre}

	We define the analogue of $G_{AF, \mathbb{R}^4}$ for $\mathbb{CP}^2\setminus \{ \text{pt}\}$. 
	
	\begin{definition} \label{G}
		Define $G_{AF, \mathbb{CP}^2\setminus\{\mathrm{pt}\}}$ to be the class of all complete metrics $g$ of the form (\ref{eq:1}) on $\mathbb{CP}^2\setminus\{ \mathrm{pt}\}$ such that
		\begin{enumerate}[i)]
			\item $u \leq 1$,
			\item $b_{s},c_{s}\geq 0$,
			\item $\sup_{p\in \mathbb{CP}^2\setminus\{ \mathrm{pt}\}} (d_g(bolt,p))^{2+\epsilon} \left\lvert Rm_{g} \right\rvert _{g}(p) < \infty, \text{ for some } \epsilon>0$.
		\end{enumerate}
	\end{definition}
	From the explicit formula (\ref{TB}), it can be easily verified that $g_{Bolt} \in G_{AF, \mathbb{CP}^2\setminus\{\mathrm{pt}\}}$.
	We note that metrics in $G_{AF, \mathbb{CP}^2\setminus\{\mathrm{pt}\}}$ are complete and have bounded curvature, and so the following theorem gives the existence of a unique maximal solution to the Ricci flow with initial metric in $G_{AF, \mathbb{CP}^2\setminus\{\mathrm{pt}\}}$. 
	We will now show that $G_{AF, \mathbb{CP}^2-\{\text{pt}\}}$ is preserved under the Ricci flow. First we need some curvature estimates. For the rest of Section \ref{GAFsection} we assume that we have a $U(2)$-symmetric Ricci flow on $(\mathbb{CP}^2\setminus \{\mathrm{pt}\}$, with $b,c$ defined by \ref{rf}.
	
	\begin{lemma}[\cite{App1} Lemmas 3.3, 3.4 and 5.1] \label{lemma}
		Let $(\mathbb{CP}^2\setminus \{\mathrm{pt}\},g)$ satisfy $\lvert Rm_{g} \rvert _{g} \leq K$, with $g$ of the form (\ref{eq:1}). Then, as $s\rightarrow \infty$,
		\begin{equation*}
			\lvert c_{s} \rvert , \left\lvert \frac{c}{b}b_{s} \right\rvert , \left\lvert u^2 \right\rvert = O\left(e^{2\sqrt{K}s}\right).
		\end{equation*}
		Furthermore,
		\begin{enumerate}[i)]
			\item $b^2\geq \frac{1}{K}$,
			\item $\frac{b_{s}^2}{b^2} \leq 5K$, 
			\item $\frac{u^2}{b^2}\leq \frac{5K}{3}.$
			\item $-2\sqrt{K}<\frac{c_{s}}{c}<\frac{1}{s}+\sqrt{K}.$
		\end{enumerate}
	\end{lemma}
	\begin{proof}
		The contents of the lemma is contained in statements of \cite[Lemmas 3.3, 3.4, 5.1]{App1} when $k=1$.
	\end{proof}
	\subsubsection{Preservation of $u\leq 1$ and $b_{s},c_{s}\geq 0$}
	
	\begin{lemma} \label{u}
		Let $(\mathbb{CP}^2\setminus \{\mathrm{pt}\},g(t))$ for $[0,T]$ be a Ricci flow with bounded curvature. If $g(0)$ satisfies $u \leq 1$, then $u \leq 1$ for all $t \in [0,T]$. 
	\end{lemma}
	\begin{proof}
		This is exactly \cite[Lemma 5.2]{App1} when $k=1$.
	\end{proof}
	
	\begin{lemma}\label{b_sc_s}
		Let $(\mathbb{CP}^2\setminus \{\mathrm{pt}\},g(t))$ for $[0,T]$ be a Ricci flow with bounded curvature. If $g(0)$ satisfies $b_{s},c_{s} \geq 0$, then $b_{s},c_{s} \geq 0$ for all $t \in [0,T]$. 
	\end{lemma}
	\begin{proof}
		This is exactly \cite[Lemma 5.6]{App1} when $k=1$.
	\end{proof}
	
	\subsubsection{Preservation of the curvature decay condition}
	We adapt the argument used in the proof of \cite[Lemma 3.1]{FraNUT} to prove the curvature decay condition:
	\begin{equation*}
		\sup_{p\in \mathbb{CP}^2\setminus\{ \text{pt}\}} (d_g(bolt,p))^{2+\epsilon} \left\lvert Rm_{g} \right\rvert _{g}(p(s)) < \infty, \text{ for some } \epsilon>0,
	\end{equation*}
	is preserved under the Ricci flow. 	
	
	We say that $\phi \in C^{\infty}(\mathbb{CP}^2\setminus \{\mathrm{pt}\})$ is a distance-like function on $\mathbb{CP}^2\setminus \{\mathrm{pt}\}$ if there exists $p_0\in M$ and $C>0$ such that 
	\begin{align*}
		C^{-1}(d_{g_{0}}(p,p_{0})+1)&\leq \phi(p) \leq C(d_{g_{0}}(p, p_{0})+1),\\
		\lvert \nabla \phi \rvert_{g_{0}} &\leq C,\\
		\text{Hess}_{g_{0}}(\phi)&\leq Cg_{0}.
	\end{align*}

	\begin{theorem} \label{curvdecaypre}
		Let $(\mathbb{CP}^2\setminus \{\mathrm{pt}\},g(t))_{0 \leq t <T}$ be the maximal Ricci flow solution starting at some $g_{0} \in G_{AF, \mathbb{CP}^2\setminus \{\mathrm{pt}\}}$, $p_{0}$ be a point on the bolt, and $\epsilon >0$ be such that
		\begin{equation}
			\sup_{p\in \mathbb{CP}^2\setminus\{ \mathrm{pt}\}} (d_{g(0)}(p_{0},p))^{2+\epsilon} \lvert Rm_{g(0)} \rvert _{g(0)}(p) < \infty.
		\end{equation}
		For $T'< T$, there exists $\alpha(T')$ such that
		\begin{equation*}
			\sup_{p\in \mathbb{CP}^2\setminus\{ \mathrm{pt}\}} (d_{g_{t}}(p_{0},p))^{2+\epsilon} \left\lvert Rm_{g(t)} \right\rvert _{g(t)}(p) \leq \alpha(T')
		\end{equation*}
		for all $t \in [0,T']$. 
	\end{theorem}
	
	\begin{proof}
		We will prove Theorem \ref{curvdecaypre} by using \cite[Proposition B.10]{Lott}.
		
		Write $$g(0)=ds_{0}^2+ b(s_{0})^2(\sigma_{1}^2+\sigma_{2}^2)+ c(s_{0})^2 \sigma_{3}^2$$ for some $b,c$ satisfying Conditions \ref{0(-1)cond}. Let $\partial_{s_{0}}, X_{1}, X_{2}, X_{3}$ be dual to $ds_{0}, \sigma_{1}, \sigma_{2}, \sigma_{3}$. Define $\phi \in C^{\infty}(M)$ by $\phi(s_{0})= \sqrt{s_{0}^2+1}$. Since $\phi$ is even as a function of $s_{0}$, it is indeed smooth on $M$. We have, 
		\begin{align*}
			\lvert \nabla_{g_{0}} \phi \rvert_{g_{0}} &= \lvert \partial_{s_{0}} \phi \rvert= \left\lvert \frac{s_0}{\sqrt{s_{0}^2+1}}\right\rvert \leq 1,\\
			\text{Hess}_{g_{0}}\phi (\partial_{s_{0}}, \partial_{s_{0}})= \partial_{s_{0}}^2 \phi&= \frac{1}{(s^2_{0}+1)^{\frac{3}{2}}}\leq 1.
		\end{align*}
		Also, 
		\begin{equation*}
			\text{Hess}_{g_{0}}\phi\left(\frac{X_{1}}{\lvert X_{1} \rvert_{g_{0}}}, \frac{X_{1}}{\lvert X_{1} \rvert_{g_{0}}}\right)=\text{Hess}_{g_{0}}\phi\left(\frac{X_{2}}{\lvert X_{2} \rvert_{g_{0}}}, \frac{X_{2}}{\lvert X_{2} \rvert_{g_{0}}}\right)= \frac{s_{0}b_{s_{0}}}{b\sqrt{s_{0}^2+1}} \leq \frac{b_{s_{0}}}{b},
		\end{equation*}
		is bounded by Lemma \ref{lemma}. Similarly, 
		\begin{equation*}
			\text{Hess}_{g_{0}}\phi(\frac{X_{3}}{\lvert X_{3} \rvert_{g_{0}}}, \frac{X_{3}}{\lvert X_{3} \rvert_{g_{0}}})= \frac{s_{0}c_{s_{0}}}{\sqrt{s_{0}^2+1}c} \leq \frac{s_{0}c_{s_{0}}}{c},
		\end{equation*}
		is bounded by Lemma \ref{lemma} and that $c_{s_0}\rightarrow1$ as $s\rightarrow 0$. Therefore $\phi$ is a smooth distance-like function on $(\mathbb{CP}^2\setminus \{\mathrm{pt}\},g(0))$. Since $b_{s_{0}},c_{s_{0}}\geq 0$, we have $\text{Hess}_{g_{0}}\phi \geq 0$. Thus if we choose $B_{0}=0$ in \cite[Proposition B.10]{Lott}, the proof is complete. 
	\end{proof}
	
	Lemmas \ref{u}, \ref{b_sc_s} and Theorem \ref{curvdecaypre} show:
	
	\begin{corollary}\label{GAFCpre}
		Let $(M,g(t))_{0 \leq t <T}$ be the maximal Ricci flow solution starting at some $g_{0} \in G_{AF, \mathbb{CP}^2\setminus \{\mathrm{pt}\}}$. Then $g(t) \in G_{AF, \mathbb{CP}^2\setminus\{\mathrm{pt}\}}$ for all $t\in[0,T)$.
	\end{corollary}

	\subsubsection{Conservation of mass}
	Lemma \ref{masswelldfiendforFra} showed the notion of the mass of a metric in $G_{AF, \mathbb{R}^4}$ is well defined. Here we will show the same for metrics in $G_{AF, \mathbb{CP}^2\setminus\{\text{pt}\}}$.
	The proof of \cite[Lemma 2.4]{FraNUT} for metrics on $\mathbb{R}^4$ only relies on the properties of the metric at infinity and so carries over exactly to prove the following analogue on $\mathbb{CP}^2\setminus\{\text{pt}\}$.
	
	\begin{lemma} \label{asmptoticform}
		Let $g$ be a metric on $\mathbb{CP}^2\setminus\{\mathrm{pt}\}$ of the form (\ref{eq:1}) such that for some $\epsilon>0$,
		\begin{equation}
			\sup_{p\in \mathbb{CP}^2\setminus\{ \mathrm{pt}\}} \left(d_{g}(bolt,p)\right)^{2+\epsilon} \left\lvert Rm_{g} \right\rvert _{g}(p) < \infty.
		\end{equation}
		Then one of the following holds:
		\begin{enumerate}[i)]
			\item $b_{s} \rightarrow 2, c_{s} \rightarrow 0, c \rightarrow m_{g}^{-1} \in (0, \infty)$ as $s\rightarrow \infty$.
			\item $b_{s} \rightarrow 1, c_{s} \rightarrow 1$ as $s\rightarrow \infty.$
		\end{enumerate}
	\end{lemma}

	\begin{definition}
		For a metric $g\in G_{AF, \mathbb{CP}^2\setminus\{ \mathrm{pt}\}}$, we define $m_{g}=\lim_{s \rightarrow \infty}\frac{1}{c(s)}$, where in the case of ii) of Lemma \ref{asmptoticform}, $m_g\coloneqq0$.
	\end{definition}
	Again, the proof of mass conversation for Ricci flow in $G_{AF,\mathbb{R}^4}$ \cite[Corollary 3.2]{FraNUT} only relies on the properties of the metric at infinity and so also proves the conservation of mass for Ricci flow in $G_{AF, \mathbb{CP}^2\setminus\{\mathrm{pt}\}}$.
	\begin{corollary} \label{massconv}
		Let $(M=\mathbb{CP}^2\setminus\{\mathrm{pt}\},g(t))_{0 \leq t <T}$ be the maximal Ricci flow solution starting at some $g_{0} \in G_{AF, \mathbb{CP}^2\setminus\{\mathrm{pt}\}}$ with mass $m_{g_{0}}$. Then $m_{g(t)}=m_{g_0}$ for all $0 \leq t <T$. 
	\end{corollary}

	\subsection{Curvature bounds in $G_{AF, \mathbb{CP}^2\setminus\{ \mathrm{pt}\}}$}
	The purpose of this subsection is to prove the following:
	\begin{theorem}\label{curvblowup}
		Let $(\mathbb{CP}^2\setminus\{ \mathrm{pt}\},g(t))$ for $[0,T)$ be a complete Ricci flow with $g(0) \in G_{AF, \mathbb{CP}^2\setminus\{ \mathrm{pt}\}}$ with a finite time singularity at $t=T$ and the curvature blowing up at the bolt. Then there exists an $0\leq R<\infty$ (wrt to $g(0)$) 
		$$\{p\in M \text{ such that } \sup_{[0,T)}\left\lvert Rm_{g(t)} \right\rvert_{g(t)}(p)=\infty\}= \{p\in M \text{ such that } d_{g(0)}(bolt,p)\in I_R\},$$
		where $I_R$ is the either the interval $[0,R]$ or $[0,R)$.
	\end{theorem}
	
	\begin{remark}
		The fact that the curvature does not blow up everywhere follows simply from the pseudolocality result \cite{pseudolocality}.  Theorem \ref{curvblowup} gives us a more precise description of the set on which the curvature blows up, which will be sufficient to prove Theorem \ref{fullthm}.
	\end{remark}

	The next proposition is main result used in the proof of $\ref{curvblowup}$ and will show that, for Ricci flow in $G_{AF, \mathbb{CP}^2-\{ \mathrm{pt}\}}$ if the curvature blows up in finite time at a point $p$, then $c(p)$ must become arbitrarily small.

	\begin{proposition}\label{cto0}
		Let $(\mathbb{CP}^2\setminus\{\mathrm{pt}\},g(t))$ for $[0,T]$ be a complete Ricci flow with $g(0) \in G_{AF, \mathbb{CP}^2-\{ \mathrm{pt}\}}$. Then there exists $C>0$ such that 
		\begin{equation} \label{cest}
			c^2(s)\left\lvert \mathrm{Rm}_{g(t)} \right\rvert_{g(t)} (s) \leq C.
		\end{equation}
		all $t \in [0,T]$ and $s \geq 0$.
	\end{proposition}
	
	Proposition \ref{cto0} will follow from Lemmas \ref{b_sc_sbound}, \ref{k_12} and \ref{k_01} which will be proven by adapting the arguments used in \cite[Lemma 4.3, Lemma 4.10]{FraBer}, \cite[Lemma 7]{Dan1}. By Corollary \ref{GAFCpre}, $G_{AF, \mathbb{CP}^2-\{ \mathrm{pt}\}}$ is preserved and so properties $b_s,c_s\geq 0$ and $u\leq 1$ hold for all time, and will be used throughout the proofs of Lemmas \ref{b_sc_sbound}, \ref{k_12} and \ref{k_01} without comment.

	\begin{lemma}\label{b_sc_sbound}
		Let $(\mathbb{CP}^2\setminus\{\mathrm{pt}\},g(t))$ for $[0,T)$ be a complete Ricci flow with $g(0) \in G_{AF, \mathbb{CP}^2-\{ \mathrm{pt}\}}$. Then there exists $C(g(0))>0$ such that 
		\begin{equation*}
			b_{s},c_{s}\leq C.
		\end{equation*}
	\end{lemma}
	
	\begin{proof}
		First we prove the estimate $b_{s}\leq C$.
		The evolution equation for $b_{s}$ is 
		\begin{equation}\label{b_sevolution}
			\partial_{t} b_{s}=b_{sss}+\frac{c_{s}}{c}b_{ss}+\frac{1}{b^2}\left(-\frac{c_{s}^2b_{s}}{u^2}+4uc_{s}-6u^2b_{s}-b_{s}^3+4b_{s}\right).
		\end{equation}
		By Lemma \ref{asmptoticform}, $\lim_{s\rightarrow \infty} b_{s}(s)=2$. Also, $b_{s}(0)=0$ by the boundary Condition \ref{0(-1)cond}. Suppose that $b_{s}$ becomes unbounded as $t\rightarrow T$. Let $\overline{\alpha}>\sup_M\lvert b_{s}\rvert(\cdot,0)>0$ be some large constant to be chosen below. Then there exists a maximum point $(p_{0}, t_0)\in \mathbb{CP}^2-\{ \text{pt}\}\times (0,T)$ not contained in the bolt for $b_{s}(\cdot, t_{0})$ where $b_{s}=\overline{\alpha}$ for the first time. Using $\frac{c}{b}\leq1$ we evaluate (\ref{b_sevolution}) at $(p_{0},t_{0})$,
		\begin{equation*}
			\partial_{t}b_{s}(p_{0},t_{0}) \leq \frac{1}{b^2}(4\overline{\alpha}-\overline{\alpha}^3-\overline{\alpha}c_{s}^2-6\overline{\alpha}u^2+4uc_{s}).
		\end{equation*}
		Choose $\overline{\alpha}>\text{max}\{\sup_M\lvert b_{s}\rvert(\cdot,0),2 \}$. Then the polynomial 
		\begin{equation*}
			-\overline{\alpha}c_{s}^2+4uc_{s}+(4\overline{\alpha}-\overline{\alpha}^3-6\overline{\alpha}u^2)
		\end{equation*}
		in $c_{s}$ has discriminant:
		\begin{equation*}
			16u^2+4\overline{\alpha}(4\overline{\alpha}-\overline{\alpha}^3-6\overline{\alpha}u^2)=8u^2(2-3\overline{\alpha}^2)+4\overline{\alpha}^2(4-\overline{\alpha}^2)<0.
		\end{equation*}
		if $\overline{\alpha}>2$. Therefore, $\partial_{t}b_{s}(p_{0},t_{0})<0$, giving us a contradiction. Therefore $b_{s}\leq C$ for some $C(g(0))>0$ along the flow. 
		
		We now carry out a very similar proof to prove the estimate $c_{s}\leq C$. The evolution equation for $c_{s}$ is
		\begin{equation}\label{c_sevolution}
			\partial_{t}c_{s}=c_{sss}+\left(2\frac{b_{s}}{b}-\frac{c_{s}}{c}\right)c_{ss}+\frac{1}{b^2}(-2c_{s}b_{s}^2-6u^2c_{s}+8u^3b_{s}).
		\end{equation} 
		By Lemma \ref{asmptoticform} we know that $\lim_{s\rightarrow \infty}c_{s}(s)=0$. Also $c_{s}(0)=0$ by boundary Condition \ref{0(-1)cond}. Suppose that $c_{s}$ becomes unbounded as $t\rightarrow T$. Let $\overline{\alpha}>\sup_M\lvert c_{s}\rvert(\cdot,0)>0$ be some large constant to be chosen below. Then there exists a maximum point $(p_{0}, t_0)\in \mathbb{CP}^2-\{ \text{pt}\}\times (0,T)$ with $p_0$ not contained in the bolt for $c_{s}(\cdot, t_{0})$ where $c_{s}=\overline{\alpha}$ for the first time. Using $u\leq1$ we evaluate (\ref{c_sevolution}) at $(p_{0},t_{0})$,
		\begin{equation*}
			\partial_{t}c_{s}(p_{0},t_{0})\leq\frac{1}{b^2}\left(-\overline{\alpha}(6u^2+2b_{s}^2)+8u^3b_{s}\right).
		\end{equation*}
		Choose $\overline{\alpha}>\text{max}\{\sup_M\lvert c_{s}\rvert(\cdot,0),2 \}$. Then the polynomial 
		\begin{equation*}
			-2\overline{\alpha}b_{s}^2+8u^3b_{s}-2\overline{\alpha}u^2.
		\end{equation*}
		in $b_{s}$ has discriminant:
		\begin{equation*}
			64u^6-16u^2\overline{\alpha}^2=16u^2(4-\overline{\alpha}^2)
		\end{equation*}
		if $\overline{\alpha}>2$. Therefore, $\partial_{t}c_{s}(p_{0},t_{0})<0$, giving us a contradiction. Therefore $c_{s}\leq C$ for some $C(g(0))>0$ along the flow.
	\end{proof}
	
	\begin{lemma}\label{k_12}
		Let $(\mathbb{CP}^2\setminus\{ \mathrm{pt}\},g(t))$ for $[0,T)$ be a complete Ricci flow with $g(0) \in G_{AF, \mathbb{CP}^2-\{ \mathrm{pt}\}}$. Then there exists a $C(g(0))>0$ such that $c^2\lvert K_{12} \rvert,c^2\lvert K_{13} \rvert,c^2\lvert K_{23} \rvert \leq C$.
	\end{lemma}
	\begin{proof}
		Let $C>0$ be the constant given by Lemma \ref{b_sc_sbound}, and $\overline{C}(g(0))>0$ a generic constant depending only on $g(0)$. Using (\ref{sectcurv}, it follows that
		\begin{equation*}
			c^2\lvert K_{12} \rvert =\frac{c^2}{b^2}(4-3\frac{c^2}{b^2}-b_{s}^2) \leq 4+3+C^2 \leq\overline{C}.
		\end{equation*}
		Also,
		\begin{equation*}
			c^2\lvert K_{13}\rvert=c^2\lvert K_{23}\rvert =\frac{c^4}{b^4}-b_{s}c_{s}\leq 1+C^2\leq \overline{C}.
		\end{equation*}
	\end{proof}
	Obtaining control of the remaining parts of the curvature will require a involved argument.
	\begin{lemma}\label{k_01}
		Let $(\mathbb{CP}^2\setminus\{ \mathrm{pt}\},g(t))$ for $[0,T)$ be a complete Ricci flow with $g(0) \in G_{AF, \mathbb{CP}^2-\{ \mathrm{pt}\}}$. Then there exists a constant $C(g(0))>0$ such that 
		\begin{equation*}
			c^2\lvert K_{01}\rvert ,c^2\lvert K_{02} \rvert, c^2\lvert K_{03} \rvert \leq C.
		\end{equation*}
	\end{lemma}
	
	\begin{proof}
		We first prove the estimate $\lvert cb_{ss} \rvert \leq C$ along the flow, for some $C>0$ depending only on $g(0)$. We do this by finding a lower bound on the quantity $f_{-}\eqqcolon cb_{ss}-\mu b_{s}^2-\nu c_{s}^2$ and an upper bound for $f_{+}\eqqcolon cb_{ss}+\mu b_{s}^2+\nu c_{s}^2$, where $\mu, \nu>0$ are to be chosen below.
		The evolution equation for $b_{ss}$ is 
		\begin{align*}
			\partial_{t}b_{ss}&=b_{ssss}+\frac{c_{s}}{c}b_{sss}+\frac{4cc_{ss}}{b^3}-\frac{2c_{s}^2b_{ss}}{c^2}+\frac{2c_{s}^3b_{s}}{c^3}\\
			&\hspace{4mm}-\frac{24cc_{s}b_{s}}{b^4}+\frac{4c_{s}^2}{b^3}-\frac{2c_{s}c_{ss}b_{s}}{c^2}-\frac{6c^2b_{ss}}{b^4}+\frac{24c^2 b_{s}^2}{b^5}\\
			&\hspace{4mm}-\frac{2b_{ss}^2}{b}+\frac{4b_{ss}}{b^2}-\frac{8b_{s}^2}{b^3}-\frac{3b_{s}^2b_{ss}}{b^2}.
		\end{align*}
		Using the evolution equations (\ref{b_sevolution}), (\ref{c_sevolution}) for $b_{s},c_{s}$, we get the evolution equation for $f_{-}$
		\begin{equation}\label{f-evolution}
			\begin{split}
				\partial_{t}(f_{-})&=(f_{-})_{ss}-\frac{c_{s}}{c}(f_{-})_{s}-cb_{ss}\left(\frac{c_{s}^2}{c^2}+\frac{8u^2}{b^2}+3\frac{b_{s}^2}{b^2}+\frac{2\mu b_{s}c_{s}}{c^2}-\frac{2b_{s}c_{s}}{bc}\right)\\
				&\hspace{4mm}+\frac{4c^2c_{ss}}{b^3}+\frac{2c_{s}^3b_{s}}{c^2}-\frac{24c^2c_{s}^2b_{s}}{b^4}+\frac{4c_{s}^2 c}{b^3}-\frac{2c_{s}c_{ss}b_{s}}{c}+\frac{24c^3b_{s}^2}{b^5}-\frac{8b_{s}^2c}{b^3}\\
				&\hspace{4mm}-\frac{2\mu b_{s}}{b^2}\left(-\frac{c_{s}^2b_{s}b^2}{c^2}+\frac{4cc_{s}}{c}-6u^2b_{s}-b_{s}^3+4b_{s}\right)\\
				&\hspace{4mm}-2\nu c_{s}\left(2\left(\frac{b_{s}}{b}-\frac{c_{s}}{c}\right)c_{ss}+\frac{1}{b^2}\left(-2b_{s}^2c_{s}-6u^2c_{s}+8u^3b_{s}\right)\right)\\
				&\hspace{4mm}+2(\mu -u)b_{ss}^2+2\nu c_{ss}^2+\frac{cb_{ss}}{b^2}.
			\end{split}
		\end{equation}
		By the curvature decay preservation given by Theorem \ref{curvdecaypre}, we have 
		$$\lim_{s\rightarrow \infty}\lvert cb_{ss} \rvert \leq \lim_{s\rightarrow \infty} \lvert bb_{ss}\rvert= \lim_{s\rightarrow \infty}\lvert b^2k_{01} \rvert=0.$$ Also, 
		$$\lim_{s\rightarrow \infty}b_{s}=2, \lim_{s\rightarrow \infty}c_{s}=0, \lim_{s\rightarrow 0}b_{s}=0, \lim_{s\rightarrow 0}c_{s}=1 \text{ and }  \lim_{s\rightarrow 0} c=0,$$ by Lemma \ref{asmptoticform} and Condition \ref{0(-1)cond}. Suppose that $f_{-}$ becomes unbounded below as $t\rightarrow T$. Let $\overline{\alpha}>0$ is some large constant to be chosen below. Then there exists a minimum point $(p_{0}, t_0)\in \mathbb{CP}^2-\{ \text{pt}\}\times (0,T)$ with $p_0$ not contained in the bolt for $f_{-}(\cdot, t_{0})$ where $f_{-}=-\overline{\alpha}$ for the first time. Since $b_{s},c_{s}$ are bounded above by Lemma \ref{b_sc_sbound} we can choose $\overline{\alpha}$ large enough so that $cb_{ss}\leq -\frac{\overline{\alpha}}{2}$. 
		With the assumption that $\mu\geq 1$, evaluating (\ref{f-evolution}) at $(p_{0},t_{0})$ gives,
		\begin{align*}
			\partial_{t}f_{-}(p_{0},t_{0})&\geq \frac{\overline{\alpha}}{2}\left(\frac{c_{s}^2}{c^2}+\frac{8u^2}{b^2}+\frac{3b_{s}^2}{b^2}+\frac{2b_{s}c_{s}}{c}\left(\frac{\mu}{c}-\frac{1}{b}\right)\right)+\frac{4c^2 c_{ss}}{b^3}-\frac{24c^2 c_{s}^2 b_{s}}{b^4}-\frac{2c_{s}c_{ss}b_{s}}{c}\\
			&\hspace{4mm}-\frac{8b_{s}^2c}{b^3}-\frac{2\mu b_{s}}{b^2}\left(\frac{4cc_{s}}{b}+4b_{s}\right)-2\nu c_{s}\left(2\left(\frac{b_{s}}{b}-\frac{c_{s}}{c}\right)c_{ss}+\frac{8u^3b_{s}}{b^2}\right)\\
			&\hspace{4mm}+2(\mu-u)b_{ss}^2+2\nu c_{ss}^2+\frac{4cb_{ss}}{b^2}.
		\end{align*}
		Let $\alpha>0$ denote a uniform constant that may change from line to line. Using Young's inequality, $u\leq 1$, and the bounds on $b_{s},c_{s}$ of Lemma \ref{b_sc_sbound} we have the following bounds
		\begin{align*}
			&\frac{4c^2c_{ss}}{b^3}\leq 8u\frac{u^2}{b^2}+\frac{c_{ss}^2}{2}\leq \alpha\frac{u^2}{b^2}+\frac{c_{ss}^2}{2},\\
			&24\frac{c^2c_{s}^2b_{s}}{b^4}=\frac{\alpha c^2}{b^3}\frac{b_{s}}{b} \leq \alpha \frac{u^2}{b^2}+\frac{b{s}^2}{2b^2},\\
			&\frac{2c_{s}c_{ss}b_{s}}{c}\leq \alpha \frac{c_{ss}c_{s}}{c}\leq \alpha \frac{c_{s}^2}{c^2}+\frac{c_{ss}^2}{c^2},\\
			&8\frac{b_{s}^2c}{b^2}= \frac{\alpha b_{s}}{b}\frac{c}{b^2}\leq \alpha\frac{b_{s}^2}{b^2}+\frac{u^2}{b^2},\\
			&8\frac{\mu c b_{s}c_{s}}{b^3}\leq 8\frac{\mu b_{s}}{b}\frac{c_{s}}{c}\leq \alpha \mu (\frac{b_{s}^2}{b^2}+\frac{c_{s}^2}{c^2}),\\
			&4\nu \frac{b_{s}^2}{b^2}c_{ss}\leq \alpha \frac{b_{s}^2}{b^2}+ \frac{c_{ss}^2}{2},\\
			&4\nu \frac{c_{s}^2c_{ss}}{c} \leq \alpha \frac{c_{s}^2}{c^2}+ \frac{c_{ss}^2}{2},\\
			&16\nu \frac{u^3 b_{s}c_{s}}{b^2}\leq \alpha \frac{u^2}{b^2}+\frac{b_{s}^2}{2b^2},\\
			&4\frac{cb_{ss}}{b^2}\leq \alpha \frac{u^2}{b^2}+\frac{b_{ss}^2}{2}.
		\end{align*}
		Therefore, for some constant $\beta(\mu ,\nu)>0$ we have
		\begin{align*}
			\partial_{t}f_{-}(p_{0},t_{0})&\geq \frac{\overline{\alpha}}{2}\left(\frac{c_{s}^2}{c^2}+\frac{8u^2}{b^2}+\frac{3b_{s}^2}{b^2}+\frac{2b_{s}c_{s}}{c}\left(\frac{\mu}{c}-\frac{1}{b}\right)\right)\\
			&\hspace{4mm}-\beta(\mu,\nu)\left(\frac{c_{s}^2}{c^2}+\frac{8u^2}{b^2}+\frac{3b_{s}^2}{b^2}+\frac{2b_{s}c_{s}}{c}\left(\frac{\mu}{c}-\frac{1}{b}\right)\right)\\
			&\hspace{4mm}+(\mu-\alpha)b_{ss}^2+(\nu-\alpha)c_{ss}^2>0,
		\end{align*}
		if we choose $\mu,\nu$ and $\bar{\alpha}$ large enough. We have a contradiction. Therefore $f_{-} \geq C(g(0))$ along the flow. 
		
		We now show $f_{+}\leq C(g(0))$.  
		Similarly, using Young's inequality we estimate $f_{+}$ at a maximum point $(p_{0},t_{0})$,
		\begin{align*}
			\partial_{t}f_{+}(p_{0},t_{0})&\leq \-cb_{ss}\left(\frac{c_{s}^2}{c^2}+\frac{8u^2}{b^2}+\frac{3b_{s}^2}{b^2}+\frac{2b_{s}c_{s}}{c}\left(\frac{\mu}{c}-\frac{1}{b}\right)\right) + 4\frac{c^2c_{ss}}{b^3}+\frac{2c_{s}^3b_{s}}{c^2}+\frac{4c_{s}^2c}{b^3}-2\frac{c_{s}c_{ss}b_{s}}{c}\\
			&\hspace{4mm}+\frac{24c^3b_{s}^2}{b^5}+\frac{2\mu b_{s}}{b^2}\left(\frac{4cc_{s}}{b}+4b_{s}\right)+2\nu \left(2\left(\frac{b_{s}}{b}-\frac{c_{s}}{c}\right)c_{ss}+\frac{8u^3b_{s}}{b^2}\right)\\
			&\hspace{4mm}-2(\mu+u)b_{ss}^2-2\nu c_{ss}^2+\frac{4cb_{ss}}{b^2}\\
			&\leq -\frac{\overline{\alpha}}{2}\left(\frac{c_{s}^2}{c^2}+\frac{8u^2}{b^2}+\frac{3b_{s}^2}{b^2}+\frac{2b_{s}c_{s}}{c}\left(\frac{\mu}{c}-\frac{1}{b}\right)\right)\\
			&\hspace{4mm}+\beta(\mu, \nu)\left(\frac{c_{s}^2}{c^2}+\frac{8u^2}{b^2}+\frac{3b_{s}^2}{b^2}+\frac{2b_{s}c_{s}}{c}\left(\frac{\mu}{c}-\frac{1}{b}\right)\right)\\
			&\hspace{4mm}-(\mu-\alpha)b_{ss}^2-(\nu -\alpha)c_{ss}^2<0,
		\end{align*}
		for $\mu, \nu$ and $\bar{\alpha}$ chosen large enough.
		
		Therefore $\lvert cb_{ss} \rvert \leq C(g(0))$. Thus $\lvert c^2K_{01}\rvert = \lvert c^2K_{02} \rvert =\lvert ucb_{ss} \rvert \leq \lvert cb_{ss} \rvert \leq C(g(0))$. 
		By bounding $\lvert cc_{ss}-\mu b_{s}^2-\nu c_{s}^2 \rvert$, a very similar argument gives a bound $\lvert c^2 K_{03} \rvert \leq C(g(0))$. 
	\end{proof}
	
	\begin{proof}[Proof of Proposition \ref{cto0}]
		Using (\ref{sectcurv}), we see that the contribution of $R_{0123}$ to $\lvert Rm \rvert^2$ is:
		\begin{equation}\label{R0123}
			\frac{1}{b^4c^2}\left(c^2 R_{0123}\right)^2=\left\lvert \frac{c^4(c_{s}-ub_{s})}{b^4} \right\rvert \leq C(g(0)).
		\end{equation}
		From equation (\ref{R0123}) and Lemmas \ref{k_12}, \ref{k_01} we have
		\begin{equation*}
			c^2\left\lvert Rm_{g(t)} \right\rvert_{g(t)} \leq C(g(0)),
		\end{equation*}
		along the flow.
	\end{proof}
	
	\begin{corollary}\label{climit}
		Let $(\mathbb{CP}^2\setminus\{ \mathrm{pt}\},g(t))$ for $[0,T)$ be the maximal complete Ricci flow with $g(0) \in G_{AF, \mathbb{CP}^2\setminus\{ \mathrm{pt}\}}$ with $T<\infty$. Then for every $r\geq0$, the limit $\lim_{t\uparrow T}c(r,t)$ exists and is finite.
	\end{corollary}
	
	\begin{proof}
		Lemmas \ref{b_sc_sbound} and \ref{k_01} give the bound
		\begin{equation*}
			\lvert \partial_{t}(c^2)\rvert \leq 2\lvert cc_{ss} \rvert +4\left\lvert ub_{s}c_{s}-u^4 \right\rvert \leq C(g(0)).
		\end{equation*}
		Since $T<\infty$, $c^2(r,t)$ and so $c(r,t)$ must have a limit as $t\rightarrow T$.
	\end{proof}

	\begin{proof}[Proof of Theorem \ref{curvblowup}]
		Proposition \ref{cto0} and Corollary \ref{climit} show that if the curvature blows up at a point $r$ then $\lim_{t \rightarrow T}  c(r,t)=0$.
		The condition $c_{s}\geq 0$ of Definition \ref{G} means that if there is a finite time blow up at time $T$, at some $R\geq0$, then $\lim_{t \rightarrow T}  c(r,t)=0$ for all $r\leq R$. But Corollary \ref{massconv} keeps $\lim_{r \rightarrow \infty}c(r,t)>0$ constant along the flow. So the set on which the curvature blows up are all points with $s \leq R$ or $s<R$ for some $0\leq R <\infty$.
	\end{proof}
	
	\subsection{Restatement of Theorem \ref{fullthm}}
	Now the classes of metrics $G_{AF, \mathbb{R}^4},G_{AF, \mathbb{CP}^2\setminus\{ \mathrm{pt}\}}$ and their masses have been defined, the following slightly stronger version of Theorem \ref{fullthm} can be stated. Using the concept of the mass of metrics in the special class $G_{AF, \mathbb{R}^4}$ of metrics we will be able to say, up to a diffeomorphism, which member of the Taub-NUT family the post surgery flow converges to.
	\begin{theorem} \label{fullthm2}
		There is a Ricci flow $G(t)$ with surgery such that
		\begin{enumerate}[i)]
			\item The flow $G(t)$ on $0 \leq t< T$ is a solution in $G_{AF, \mathbb{CP}^2-\{ \mathrm{pt}\}}$ with $G(0)$ of mass $m$, and $G(t_0)=g_{\mathrm{Bolt}}$ outside some compact set of $\mathbb{CP}^2\setminus\{\text{pt}\}$..
			\item A finite time singularity is formed at time $T$, with the bolt shrinking to zero size and the curvature blows up exactly on either the bolt or some tubular neighbourhood of the bolt, or the closure of such a neighbourhood.
			\item The singularity at time $T$ is of Type I and is modelled on the K\"{a}hler shrinking Ricci soliton FIK.
			\item Surgery can performed near the bolt to produce a metric $\tilde{G} \in G_{AF, \mathbb{R}^4}$, with the local topology change $\mathbb{CP}^2\setminus\{ \mathrm{pt}\} \rightarrow \mathbb{R}^4$.
			\item Then the Ricci flow $G(t)\in G_{AF, \mathbb{R}^4}$ starting at $\tilde{G}$ is immortal and converges to a member of the Taub-NUT family with mass $m$ in the pointed Cheeger-Gromov sense as $t \rightarrow \infty$.
		\end{enumerate}
	\end{theorem}

	\section{The setup of the Wa{\.z}ewski box argument }\label{setupsection}
	This section will very closely follow the set up of the Wa{\.z}ewski box argument of \cite{Stol1}, which will be used to produce Ricci flow in $G_{AF, \mathbb{CP}^2\setminus\{ \mathrm{pt}\}}$ which encounters a finite time singularity modelled on FIK. 
	The box argument is a topological one: A family of Ricci flows $G_{\mathbf{p}}(t)$ with initial data parametrised by $\mathbf{p} \in \overline{B}_{R} \subset \mathbb{R}^K$, for some $R>0$ and $K\geq 2$ will be constructed. The 'box', roughly speaking, will be defined as a collection of metrics required to be within a time dependent distance from $g_{\mathrm{FIK}}(t)$. This distance becomes tends to zero as $t$ tends to 1, the time at which the curvature of $g_{\mathrm{FIK}}(t)$ blows up. If, for all $\mathbf{p} \in \overline{B}_R$, $g_{\mathbf{p}}(t)$ leaves the box before $t=1$, then there will exist a continuos map 
	$$\mathcal{F} : \overline{B}_{R } \to \mathbb{R}^K\setminus\{0\},$$
	whose restriction to $$\overline{A} = \left\{ \mathbf p \in \mathbb{R}^K : \frac{R}{2} \leq \lvert \mathbf p \rvert \leq R \right\}$$
	is homotopic to the identity $\text{Id} : \overline{A} \to \overline{A} \subset \mathbb{R}^K \setminus \{ 0 \}$. This implies that there is a map $S^{K-1}\rightarrow S^{K-1}$ which is null-homotopic, and so $S^{K-1}$ is homotopy equivalent to the topological space with one point. We have a contradiction. 
	
	The set up of the box argument used in this paper will only differ from the one used in \cite{Stol1} in that we glue our shrinker into a non-compact manifold and the eigentensors will have an additional $U(2)$-symmetry. It is easy to adapt the work of Stolarski to the non-compact setting provided we glue into a Riemannian manifold of bounded curvature, as this will allow us to use the pseudolocality result \cite{pseudolocality} to bound the curvature of the Ricci flow outside some compact subset of $M$. With the non-compactness dealt with, the fact that we have the $U(2)$-symmetric basis of eigentensors of Theorem \ref{U(2)Spectrum} and our flow is $U(2)$-symmetric, the proof of the finite time singularity modelled on FIK will follow exactly as in \cite{Stol1}, only now we decompose two-tensors via our $U(2)$-symmetric basis of eigenfunctions. 
	Combining with the results of Section \ref{GAFsection}, Theorem \ref{fullthm2} will follow with little work.

	An outline of this section is a follows:
	
	First we will construct a family of metrics $G_\mathbf{0}(t_0)$ which are metrics that interpolate between FIK, its asymptotic cone and Taub-Bolt. Then we will define the metrics $G_{\mathbf{p}}(t_0)$ by 
	\begin{equation*} 
		G_{\mathbf p} (t_0 ) \coloneqq G_{\mathbf 0 }(t_0 ) 	
		+ ( 1 - t_0) \phi_{t_0}^* \left( \eta_{\gamma_0}     \sum_{j = 1}^{K } p_j  h_{j}   \right),
	\end{equation*}
	for $\mathbf{p}\in \overline{B}_R$ and $\eta_{\gamma_0}$ a bump function with support containing the zero section of $\mathbb{CP}^2\setminus\{\text{pt}\}$. Then we will check that for all enough $|\mathbf{p}|$, $G_{\mathbf{p}}(t_0)$ can be put into diagonal form (\ref{a}). Thus, the analysis of the previous section can be utilised. Then we will set up the box argument. After that we will prove the results that need adapting to our case of a non-compact manifold. Finally, we will prove Theorem \ref{fullthm2}.

	Recall that we fixed the notation $(\mathbb{CP}^2\setminus\{\text{pt}\}, \overline{g}, f)\coloneqq ( \mathbb{CP}^2\setminus\{\text{pt}\}, g_{\mathrm{FIK}}, f_{\mathrm{FIK}})$. 
	Using Theorem \ref{U(2)Spectrum}, fix a smooth $L_{f}^{2,U(2)}$-orthonormal basis $\{h_{j}\}_{j=1}^{\infty}$ of eigenmodes with corresponding eigenvalues $\{\lambda_{j}\}$ for the Weighted Lichnernowicz Laplacian $\Delta_{f}+2\overline{Rm}$. Choose $\lambda_{*}\in (-\infty,0)\setminus \{\lambda_{j}\}$ and $K\in \mathbb{N}$ such that $\lambda_{K}>\lambda_{*}>\lambda_{K+1}$.

	\subsection{Construction of $G_{\mathbf{0}}(t_{0})$}
	Recall from \ref{FIKthm} that the FIK $(\mathbb{CP}^2\setminus\{\text{pt}\},\overline{g},f)$ is asymptotic to $C_{\mathrm{FIK}}$, the cone over the sphere $\left(S^3, g_{S^3_\mathrm{FIK}} \coloneqq \frac{1}{\sqrt{2}}(\sigma_{1}^2+\sigma_{2}^2)+\frac{1}{2}\sigma_{3}^2\right)$. 
	
	Now we construct the family $G(t_0)$ of metrics which will be an interpolation between $(1-t_0)\phi_{t_0}^*\overline{g}$ and a member of the Taub-Bolt family.
	This is done by following the construction in \cite[Appendix A]{Stol1} closely. Instead of 
	glueing FIK into a cylinders we glue it into metrics in $G_{AF, \mathbb{CP}^2\setminus\{\text{pt}\}}$. All we need to really show is that we can do this while still having good enough estimates on the curvature (i.e. an analogue of \cite[Lemma A.3]{Stol1}):

	\begin{lemma}\label{conetoTB}
		For $R\geq 3$ large enough, there exists smooth increasing functions $b_{R},c_{R}:(0,\infty)\rightarrow (0,\infty)$ which satisfy
		$$b_R(s) \left\{ \begin{array}{ll}
			= \frac{s}{\sqrt[4]{2}}, 		& \text{ if } 0 < s \le R,	\\
			=  b_{\mathrm{Bolt}}(s), 	& \text{ if } 3R \le s,	\\
		\end{array} \right.$$
		$$c_R(s) \left\{ \begin{array}{ll}
			= \frac{s}{\sqrt{2}}, 		& \text{ if } 0 < s \le R,	\\
			= c_{\mathrm{Bolt}}(s), 	& \text{ if } 3R \le s,	\\
		\end{array} \right.$$
		and so that the curvature of the metric 
		$$g_R=ds^2+b_{R}^2(s)(\sigma_{1}^2+\sigma_{2}^2)+c_{R}^2(s)\sigma_{3}^2, \qquad s>\frac{R}{2},$$ 
		on $\mathbb{R}_{>0} \times S^3$ satisfies
		$$ \lvert \nabla_{g_R}^{m}Rm[g_R]\rvert_{g_R}\leq C=C(m,R),$$
		for all points with $s>\frac{R}{2}$.
		Also,
		\begin{enumerate}[i)]
			\item $\frac{c_R}{b_R}\leq 1-\delta$ \text{ for some } $\delta=\delta(R)>0$,
			\item $(b_R)_s, (c_R)_s >0$,
			\item $\sup_{s(p)>\frac{R}{2}} (s(p))^{2+\epsilon} \left\lvert Rm_{g_R} \right\rvert _{g_R}(p) < \infty, \text{ for some } \epsilon>0$.
		\end{enumerate}
		Furthermore, there exists $c=c(R)>0$ such that,
		$$Vol_{g_R}(B_{g_R}(x,1) \geq c, \qquad \text{ for all } s(x)>\frac{R}{2}.$$
	\end{lemma}
	\begin{proof}
		We can write 
		$$g_{\mathrm{Bolt}}=dr^2+b_{\mathrm{Bolt}}^2(r)(\sigma_{1}^2+\sigma_{2}^2)+ c_{\mathrm{Bolt}}^2(r)\sigma_{3}^2, \qquad r>0$$
		for some radial vector field $\partial_r$, and
		$$g_{C_{\mathrm{FIK}}}= ds^2+\frac{s^2}{\sqrt{2}}(\sigma_{1}^2+\sigma_{2}^2)+ \frac{s^2}{2}\sigma_{3}^2, \qquad s>0,$$
		for some radial vector field $\partial_s.$
		For $\alpha>0$, we have 
		\begin{align*}
			\alpha^2 g_{\mathrm{Bolt}}&=\alpha^2 (dr^2+b_{\mathrm{Bolt}}^2(r)(\sigma_{1}^2+\sigma_{2}^2)+ c_{\mathrm{Bolt}}^2(r)\sigma_{3}^2)\\
			&= dr^2+ \alpha^2 b^2_{\mathrm{Bolt}}\left(\frac{s}{\alpha}\right)(\sigma_{1}^2+\sigma_{2}^2)+ \alpha^2 c^2_{\mathrm{Bolt}}\left(\frac{s}{\alpha}\right)\sigma_{3}^2.
		\end{align*}
		
		For fixed $R>0$, $$\alpha c_{\mathrm{Bolt}}\left(\frac{s}{\alpha}\right) \rightarrow \alpha\frac{s}{\alpha}=s, \alpha b_{\mathrm{Bolt}}\left(\frac{s}{\alpha}\right) \rightarrow \alpha b_{\mathrm{Bolt}}(0),$$ as $\alpha\rightarrow \infty.$ Therefore, for $\alpha(R)>0$ large enough, 
		$$\alpha c_{\mathrm{Bolt}}\left(\frac{3R}{\alpha}\right)>\frac{R}{\sqrt{2}},~ \alpha b_{\mathrm{Bolt}}\left(\frac{3R}{\alpha}\right)>\frac{R}{\sqrt[4]{2}}.$$
		Therefore there are increasing functions $b_R, c_R$ such that 
		$$b_R(s) \left\{ \begin{array}{ll}
			= \frac{s}{\sqrt[4]{2}}, 		& \text{ if } 0 < s \le R,	\\
			=  \alpha b_{\mathrm{Bolt}}(\frac{s}{\alpha}), 	& \text{ if } 3R \le s,	\\
		\end{array} \right.$$
		$$c_R(s) \left\{ \begin{array}{ll}
			= \frac{s}{\sqrt{2}}, 		& \text{ if } 0 < s \le R,	\\
			= \alpha c_{\mathrm{Bolt}}(\frac{s}{\alpha}), 	& \text{ if } 3R \le s.	\\
		\end{array} \right.$$
		Note that $\frac{c_{\mathrm{Bolt}}}{\mathrm{b_{\mathrm{Bolt}}}}(s)\leq 1$ for all $s\geq 0$, $\frac{c_{\mathrm{Bolt}}}{\mathrm{b_{\mathrm{Bolt}}}}(0)=0$, and $\frac{c_{\mathrm{Bolt}}}{\mathrm{b_{\mathrm{Bolt}}}}(s)\rightarrow 0$ as $s\rightarrow \infty$. Thus, $(b_R)_s, (c_R)_s>0$, $\frac{c_R}{b_R}< 1-\delta$ for $s \not\in (R,3R)$ and some $\delta=\delta(R)>0$. Therefore we can interpolate in such a way that $(b_R)_s, (c_R)_s>0$, $\frac{c_R}{b_R}< 1-\delta$ for all $s>0$.	
		The curvature bound conditions and the curvature decay condition iii) clearly follows as it is true  for $g_{\mathrm{Bolt}}$. Since $R\geq 3$,  $B_{g_R}(x,1) \subset \{y: s(y)>\frac{1}{2}\}$ for all $x$ with $s(x)>\frac{R}{2}$. Therefore, the volume lower bound follows simply from the fact that
		it is true for $g_{\mathrm{Bolt}}$ as it has bounded curvature and positive injectivity radius.
	\end{proof}

	There exists a smooth family of diffeomorphisms $$\Psi_t =  \phi_t \circ \Psi : C_{\mathrm{FIK}}  \to (\mathbb{CP}^2\setminus \{\mathrm{pt}\})\setminus \{s=0\}$$ such that 
	$$ ( 1 - t) \Psi_t^*  \overline{g} \xrightarrow[t \nearrow 1]{ C^\infty_{loc} ( C(S^3), g_{C_{\mathrm{FIK}}} ) }		g_{C_{\mathrm{FIK}}}.$$
	In particular, for any $R_0$, we have uniform convergence 
	on $[R_0,\infty) \times S^3$.
	
	Let $R_1, R_2, R_3 \in \mathbb{R}$ be such that
	$$0 < R_0 < R_1 < R_2 < R_3.$$
	Let $\eta(r) : (0, \infty) \to [0,1]$ be a smooth bump function that decreases from 1 to 0 over the interval $(R_1, R_2)$.
	For all $0 \le t_0 < 1$, consider the metric $\check{G}(t_0)$ on $s>R_0$ given by 
	$$\check G(t_0)  = \eta ( 1 - t_0) \Psi^* \phi_{t_0}^* \overline{g} + ( 1 - \eta) g_{R_3}$$
	where $g_{R_3}$ is the warped product metric from lemma \ref{conetoTB}.\\

	Following the rest of \cite[Appendix A]{Stol1}, we can easily construct $G_0(t_0)$, but now with the additional property that $G(t_0) \in G_{AF, \mathbb{CP}^2\setminus\{\text{pt}\}}$. Recall the diffeomorphism $\Psi$ of Theorem \ref{FIKthm}.
	
	\begin{proposition} \label{G_0}
		Consider the metric $G(t_0)-G(\Gamma_{0}, \gamma_0, t_0)$ on $\mathbb{CP}^2\setminus \{\mathrm{pt}\}$ which consists of $( \Psi^{-1})^* \check G(t_0)$ extended by $( 1- t_0) \phi_{t_0}^* \overline{g}$. If $\frac{1}{4 \sqrt{2} } R_3 = \frac{1}{2} R_2 = R_1=\sqrt{\frac{\Gamma_{0}}{2}} \gg 1$ is sufficiently large,
		then the following holds:
		\begin{enumerate}
			\item (FIK near the bolt of $(\mathbb{CP}^2\setminus \{\mathrm{pt}\},\overline{g})$) 
			For all $0 \le t_0, t_0' < 1$, 
			$$G(t_0) = (1-t_0)\phi_{t_0}^* \overline{g} 
			\qquad \text{on } \left \{ x \in \mathbb{CP}^2\setminus \{\mathrm{pt}\} : f(x)  \leq \Gamma_{0} \right \}.$$ 
			\item (Independent of $t_0$ far from the bolt)
			$$G(t_0) = G(t_0')	\qquad \text{on } 
			\left \{ x\in \mathbb{CP}^2\setminus \{\mathrm{pt}\}:f(x)\geq 16\Gamma_0 \right\},$$
			In fact, there exists $\tilde{C}>0$ such that $G(t_0)=g_{Bolt}$ for $\{f>\tilde{C}\}$.
			\item (Convergence to cone in an intermediate region)
			$G(t_0)$ smoothly converges to $(\Psi^{-1})^* g_{C}$ in the region 
			$$\left \{x\in \mathbb{CP}^2\setminus \{\mathrm{pt}\} : \frac{\Gamma_{0}}{2}<f(x)< 32\Gamma_0 \right \}\subset \Psi(C_{R_0})$$
			as $t_0\nearrow 1$, and
			\item 	(Curvature estimates far from the bolt)
			For any $m\in\mathbb{N}$, if $0<1-t_0\ll 1$ is sufficiently small then 
			$$|\nabla^m Rm|_{G(t_0)} \leq C=C(m, \Gamma_{0}) 
			\qquad \text{on } \left\{x\in \mathbb{CP}^2\setminus \{\mathrm{pt}\} :f(x)>\frac{\Gamma_0}{2}\right\}.$$
			\item There exists $c=c(\Gamma_0)>0$ such that 
			$$Vol_{G(t_0)}(B_{G(t_0)}(x,1))\geq c \qquad \text{on } \left\{x\in \mathbb{CP}^2\setminus \{\mathrm{pt}\} :f(x)>\frac{\Gamma_0}{2}\right\}.$$
			\item $G(t_0) \in G_{AF, \mathbb{CP}^2\setminus\{\mathrm{pt}\}}$ and when written in the form (\ref{eq:1}), $\frac{c}{b} < 1-\delta$ for some $\delta>0$, and $b_s,c_s>0$ for $s>0$.
		\end{enumerate}
	\end{proposition}
	
	\begin{proof}
		Given the curvature estimates of Lemma \ref{conetoTB}, the construction of $G(t_0)$ satisfying the first five points follows immediately from the construction \cite[Appendix A]{Stol1}. 
		We now prove the final property.
		Write 
		$$(1-t_0)\phi_{t_0}^*\overline{g}= dr^2+ b_{\mathrm{FIK}}^2(t_0)(\sigma_{1}^2+\sigma_2^2)+c_{\mathrm{FIK}}^2(t_0)\sigma_{3}^2,$$
		
		$$(\Psi^{-1})^*g_{R_3}= a^2dr^2+ b^2(\sigma_{1}^2+\sigma_2^2)+c^2\sigma_{3}^2,$$
		for some radial vector field $\partial_r$.
		Then we have $(b_{\mathrm{FIK}}^2(t_0))_r, (c_{\mathrm{FIK}}^2(t_0))_r, (b^2)_r, (c^2)_r>0$ for $r>0$. Also, $\frac{c_{\mathrm{FIK}}^2(t_0)}{b_{\mathrm{FIK}}^2(t_0)}, \frac{c^2}{b^2}\leq (1-\delta)$ for $\delta>0$ independent of $t_0$ as it is true for $\overline{g}$ and $g_{R_3}$ and these inequalities are scaling and diffeomorphism invariant.
		We have
		$$G(t_0)= (\eta +(1-\eta)a^2)dr^2+ (\eta b_{\mathrm{FIK}}^2(t_0)+(1-\eta)b^2)(\sigma_{1}^2+\sigma_2^2)+ (\eta c_{\mathrm{FIK}}^2(t_0)+(1-\eta)c^2)\sigma_{3}^2,$$
		where $(\eta b_{\mathrm{FIK}}^2(t_0)+(1-\eta)b^2)_r, (\eta c_{\mathrm{FIK}}^2(t_0)+(1-\eta)c^2)_r>0$ for $r>0$. Furthermore, 
		$$\frac{\eta c_{\mathrm{FIK}}^2(t_0)+(1-\eta)c^2}{\eta b_{\mathrm{FIK}}^2(t_0)+(1-\eta)b^2}\leq (1-\delta).$$
		
		The curvature decay condition of Definition \ref{G} holds as it is true for $g_{R_3}$. 
	\end{proof}

	Now we define the family of initial metrics $G_{\mathbf{p}}(t_0)$.
	For $\gamma_0 > 0$, in \cite{Stol1} a compactly supported bump function $\eta_{\gamma_0} : M \to [0,1]$ with the following properties is shown to exist:
	\begin{enumerate}[i)]
		\item $\eta_{\gamma_0} (x) = 1$ for all $x \in \mathbb{CP}^2\setminus \{\mathrm{pt}\}$ such that $f(x) \le \frac{\gamma_0}{2 ( 1 - t_0)}$.
		\item $\text{supp} (\eta_{\gamma_0}) \subset \left \{ x \in \mathbb{CP}^2\setminus \{\mathrm{pt}\} : f(x) < \frac{\gamma_0}{1 - t_0} \right \}$.
		\item $\overline{\{x \in \mathbb{CP}^2\setminus \{\mathrm{pt}\} :  0 < \eta_{\gamma_0} (x) < 1 \}} \subset \left\{x \in \mathbb{CP}^2\setminus \{\mathrm{pt}\} :  \frac{\gamma_0}{2(1 - t_0)}    < f(x) < \frac{\gamma_0}{1 - t_0}  \right\}$. 
		\item For all $m \in \mathbb{N}$, there exists $C=C(m)$ such that if $0 < 1 - t_0 \ll1$ is sufficiently small, then 
		$$ | \overline{\nabla}^m \eta_{\gamma_0} |_{\overline{g}} \le C.$$
	\end{enumerate}
	
	Let $\overline{p}>0$. For all $\mathbf p = ( p_1, \dots , p_{K} ) \in \mathbb{R}^K$ with $| \mathbf p | \le \overline{p} (1 - t_0)^{| \lambda_*|}$, define a symmetric 2-tensor $G_{\mathbf p }(t_0) = G_{\mathbf p } ( \Gamma_0 , \gamma_0, t_0)$ on $M$ by
	\begin{equation} \label{G_p}
		G_{\mathbf p} (t_0 ) \coloneqq G_{\mathbf 0 }(t_0 ) 	
		+ ( 1 - t_0) \phi_{t_0}^* \left( \eta_{\gamma_0}     \sum_{j = 1}^{K } p_j  h_{j}   \right).
	\end{equation}

	The next couple of results will allow us to conclude that, provided $0 < 1-t_0 \ll 1$ is sufficiently small and $0 < \gamma_0=\gamma_{0}(t_0) \le 1$ is sufficiently small, $G_{\mathbf{p}}(t_0)$ is a well-defined metric on $M$. The proof of the next proposition follows exactly as the proof of [\cite{Stol1}, Lemma 4.5].\\

	\begin{proposition}[\cite{Stol1}, Proposition 8.2] \label{G_p-G-0}
		Let $\Gamma_0>0$ be sufficiently large. For all $m \in \mathbb{N}$, if
		$0 \leq t_0(m) \ll 1$ is sufficiently small and $0 < \gamma_0=\gamma_{0}(t_0) \le 1$,	
		then 
		$$| \nabla_{G_{\mathbf{0}}(t_0)}^m(G_{\mathbf p} (t_0) - G_{\mathbf 0} (t_0)) |_{G_{\mathbf 0} (t_0)} \lesssim_{\lambda_*,m}  \overline{p}  \gamma_0^{|\lambda_*|}$$
		for all $| \mathbf p | \le \overline{p} ( 1 - t_0)^{|\lambda_*|} $.
	\end{proposition}
	
	\begin{corollary}[\cite{Stol1}, Corollary 4.8]\label{welldefined}
		Let $\Gamma_0 >0$ be sufficiently large. If
		$0 < 1-t_0 \ll 1$ is sufficiently small and $0 < \gamma_0=\gamma_{0}(t_0) \le 1$ is sufficiently small,
		then $G_{\mathbf{p}}(t_0)$ is a well-defined metric on $M$.
	\end{corollary}
	
	To utilise the stability result Theorem $\ref{Fra1}$, the metric $G_{\mathbf{p}}$ will need to be in diagonal form. 
	
	\begin{proposition}\label{diagonalprop}
		Let $0\leq t_0<1$. If $\overline{p}$ is small enough, there is a choice of radial vector field $\partial_s$ such that $G_{\mathbf{p}}(t_0)$ can be put into the diagonal form (\ref{eq:1}) with respect to some $\mathbf{p}$-dependent basis $ds, \sigma_1, \sigma_2, \sigma_{3}$ as well as $G_{\mathbf{p}}(t_0) \in G_{AF, \mathbb{CP}^2\setminus\{ \mathrm{pt}\}}.$
	\end{proposition}
	
	\begin{proof}
		Fix $0\leq t_0<1$. Write $$G_{\mathbf{0}}(t_0)=dr^2+b^2(r)(\sigma_1^2+\sigma_2^2)+c^2(r)\sigma_3^2.$$
		Since $$f_{\mathbf{p}}\coloneqq( 1 - t_0) \phi_{t_0}^* \left( \eta_{\gamma_0}     \sum_{j = 1}^{K } p_j  h_{j}   \right)$$ is $SU(2)$-symmetric, with respect to the frame $dr, \sigma_{1}, \sigma_{2}, \sigma_{3}$, the tensor $f_{\mathbf{p}}$ can be thought of as a matrix depending only on $r$. As $f_{\mathbf{p}}$ also has the vector $X_3$ of \ref{U(2)} tangent to the fibres of the Hopf fibration as a Killing field, $\mathcal{L}_{X_3}f=0$, and so 
		$$f_{\mathbf{p}}=f_{00}dr^2+f_{11}(\sigma_1^2+\sigma_2^2)+f_{33}\sigma_3^2+f_{r3}(dr \sigma_3+\sigma_3 dr),$$
		for some functions $f_{00}, f_{11}, f_{33}, f_{r3}$ of $r$, and where we suppress the dependence on $\mathbf{p}$ so that the notation does not get cumbersome.

		By choosing a different radial vector field $\partial_{r'}$ so that $\partial_{r'}$ is orthogonal to the $S^3$ orbits of $SU(2)$, 
		$$G_{\mathbf{p}}(t_0)= G_{\mathbf{0}}(t_0)+f_{\mathbf{p}}= a'^2dr'^2+ b'^2(\sigma_1^2+\sigma_2^2)+ c'^2\sigma_3^2.$$ By a reparametrisation of the radial coordinate, $G_{\mathbf{p}}(t_0)$ takes the following form 
		$$G_{\mathbf{p}}(t_0)= ds^2+ b'^2(s)(\sigma_1^2+\sigma_2^2)+ c'^2(s)\sigma_3^2.$$

		Using Proposition \ref{G_p-G-0}, 
		$$|f_{\mathbf{p}}|_{G_{\mathbf{0}}(t_0)}, |\nabla f_{\mathbf{p}}|_{G_{\mathbf{0}}(t_0)}\leq C(\overline{p}) \rightarrow 0 \text{ as } \overline{p}\rightarrow 0.$$
		
		Now, the coefficients $b',c'$ are non-decreasing in $r'$ if and only if the coefficients $f_{11}+b^2,f_{33}+c^2$ are non-decreasing in $r$ since they are the coefficients of $\sigma_1^2+\sigma_2^2$ and $\sigma_3^2$ for the same metric but with respect to different radial vector fields (i.e.~ they are the same up to a reparametrisation).
		Similarly, the condition that 	 
		$\frac{c'}{b'}\leq 1$ is equivalent to the condition that $\frac{c}{b}\leq 1$. Since $f_{\mathbf{p}}$ has compact support, the curvature decay condition of Definition \ref{G} trivially holds. Thus we may assume $f_{\mathbf{p}}$ is already diagonal: 
		$$G_{\mathbf{0}}(t_0)+f_{\mathbf{p}}= diag(G_{\mathbf{0}}(t_0)+f_{\mathbf{p}})=(1+f_{00})dr^2+(f_{11}+b^2)(\sigma_1^2+\sigma_2^2)+(f_{33}+c^2)\sigma_3^2.$$
		Then, for small enough $\overline{p}$, 
		$$1+f_{00}(0)>0.$$ 
		Since $G_{\mathbf{0}}(t_0)+f_{\mathbf{p}}$ is also a metric, we must also have
		$$(\sqrt{b^2+f_{11}})_r(0)=0.$$
		Since $c_r(0)=1$, 
		$$(\sqrt{c^2+f_{33}})_r(0)>1-\delta,$$ 
		for some $\delta=\delta(\overline{p})$ for $\overline{p}$ small enough. 
		By the K\"{a}hler condition iii) of Theorem \ref{FIKthm}, 
		$$b_{rr}(0)= \lim_{s\rightarrow 0} \frac{b_s}{c}(s)= \frac{1}{b(0)}>0.$$  Since $(b^2)_r(0)=0$, it follows that $(f_{11})_r(0)=0$, and so for $\overline{p}$ small enough, we can make $(f_{11})_{rr}$ small enough near $r=0$ so that there is an $\epsilon>0$ such that $$(\sqrt{b^2+f_{11}})_r\geq 0, (\sqrt{c^2+f_{33}})_r \geq 0$$ for all $r<\epsilon.$  Since $b_r, c_r>0$ by Definition \ref{G_0}, it follows that for $A>r>\epsilon$ with $\eta_{\gamma_{0}}(A)=0$, we have $b_r(r),c_r(r)>\alpha=\alpha(\epsilon, A)$. If $\overline{p}$ is small enough, we have $$(b^2+f_{11})_r, (c^2+f_{33})_r\geq 0,$$ for all $r\geq 0$. Since, by Proposition \ref{G_0}, $\frac{c}{b}\leq 1-\delta$ for some $\delta>0$, if $\overline{p}$ small enough, $\frac{c^2+f_{33}}{b^2+f_{11}}\leq 1$ everywhere. Therefore, if $\overline{p}$ is small enough, then	$$G_{\mathbf{p}}(t_0)= ds^2+ b'^2(s)(\sigma_1^2+\sigma_2^2)+ c'^2(s)\sigma_3^2 \in G_{AF, \mathbb{CP}^2\setminus\{ \mathrm{pt}\}}.$$
	\end{proof}

	\begin{remark}
		We are being slightly lazy with the notation. The space $G_{AF, \mathbb{CP}^2\setminus\{ \mathrm{pt}\}}$ can only be defined once a unit speed geodesic intersecting all the principal orbits has been chosen. What we mean here is that for a fixed $t_0$ and $\overline{p}(t_0)$ chosen small enough, whatever  $\mathbf{p}^*$ is chosen by our box argument, we can find a unit speed geodesic, depending on $\mathbf{p}^*$, orthogonal to the principal orbits and define the corresponding space $G_{AF, \mathbb{CP}^2\setminus\{ \mathrm{pt}\}}$ of metrics so that our initial condition to the Ricci flow satisfies $G_{\mathbf{p}}(t_0)=G_{\mathbf{0}}(t_0)+f_{\mathbf{p}} \in G_{AF, \mathbb{CP}^2\setminus\{ \mathrm{pt}\}}.$ 
	\end{remark}

	\subsubsection{The flow $G_{\mathbf{p}}(t, t_{0})$}\label{assumption}

	Since $G_{\mathbf{p}}(t_0)$ has bounded curvature, by \cite{Chen1} \cite{Shi1}, we can let $G_{\mathbf{p}}(t) = G_{\mathbf{p}} ( t; \Gamma_0, \gamma_0, t_0)$ denote the maximal bounded curvature solution to the Ricci flow
	$$\partial_t G_{\mathbf{p}} (t) = -2 Ric [ G_{\mathbf{p}} (t) ]$$
	on $\mathbb{CP}^2\setminus \{\mathrm{pt}\}$ with initial data $G_{\mathbf{p}}(t_0)$ given by (\ref{G_p}) at time $t = t_0$.
	Let $T(\mathbf{p} ) = T(\mathbf{p}; \Gamma_0, \gamma_0, t_0)$ denote the minimum of $1$ and the maximal existence time of $G_{\mathbf{p}}( t ; \Gamma_0, \gamma_0 , t_0)$. The time $T(\mathbf{p})$ is chosen in this way because we aim to produce a finite time singularity at $t=1$. \\

	In \cite{Stol1} a compactly supported bump function $\eta_{\Gamma_0} : \mathbb{CP}^2\setminus \{\mathrm{pt}\} \to [0,1]$ with the following properties is shown to exist:
	\begin{enumerate}[i)]
		\item $\eta_{\Gamma_0} (x) = 1$ for all $x \in \mathbb{CP}^2\setminus \{\mathrm{pt}\}$ such that $f(x) \le  \frac{1}{2}\Gamma_0 $,
		\item $\text{supp}(\eta_{\Gamma_0}) \subset \{f< \Gamma_0  \}$,
		\item $\overline{\{x\in \mathbb{CP}^2\setminus \{\mathrm{pt}\} :0<\eta_{\Gamma_0}(x)< 1\}} \subset \{ \frac{4}{6} \Gamma_0<f<  \frac{5}{6} \Gamma_0\}	 \subset \{ \frac{1}{2} \Gamma_0 <f<\Gamma_0  \}$, 
		\item for all $m \in \mathbb{N}$, there exists $C=C(m)$ such that $\lvert {}^{(1 - t) \phi_t^* \overline{g}} \nabla^m \eta_{\Gamma_0} \rvert_{(1-t)\phi_t^* \overline{g}} \leq C $,
	\end{enumerate}
	
	We now turn our focus to a different approximate solution to the Ricci flow on $\mathbb{CP}^2\setminus \{\mathrm{pt}\}$. Set
	$$\acute G_{\mathbf{p}}(t) 
	\coloneqq \eta_{\Gamma_0} G_{\mathbf{p}}(t) + (1- \eta_{\Gamma_0}) (1 - t) \phi_t^* \overline{g}$$
	for all $t \in [t_0, T( \mathbf{p} ) )$.
	
	The author of \cite{Stol1} makes the following remark, which we will record as a lemma as it may be useful to the reader.
	\begin{lemma}[\cite{Stol1}, Remark 4.9]  \label{remark1}
		With $\acute G_{\mathbf{p}}(t) 
		\coloneqq \eta_{\Gamma_0} G_{\mathbf{p}}(t) + (1- \eta_{\Gamma_0}) (1 - t) \phi_t^* \overline{g}$ we have the following:
		\begin{enumerate}[i)]
			\item At $t = t_0$,
			$$\acute G_{\mathbf{p}}(t_0) 
			= \eta_{\Gamma_0} G_{\mathbf{0}}(t_0) 
			+ (1-\eta_{\Gamma_0} ) ( 1 - t_0)  \phi_{t_0}^* g 	
			+ \eta_{\Gamma_0}  ( 1 - t_0)  \phi_{t_0}^* \left( \eta_{\gamma_0}     \sum_{j = 1}^{K } p_j  h_{j}   \right).$$
			Recall $\mathrm{supp} (\phi_{t_0}^* \eta_{\gamma_0}) \subset \{ f < f_{\mathrm{Bolt}}+\Gamma_0 / 2 \}$ and $G_{\mathbf{0} } (t_0) = ( 1 - t_0) \phi_{t_0}^* g$ on $\{ f \le f_{\mathrm{Bolt}}+ \Gamma_0 \} \supset \mathrm{supp} (\eta_{\Gamma_0})$.
			Therefore,
			$$\acute G_{\mathbf{p}}(t_0) 
			= (1 - t_0) \phi_{t_0}^* g
			+  ( 1 - t_0)  \phi_{t_0}^* \left( \eta_{\gamma_0}    \sum_{j = 1}^{K } p_j  h_{j}  \right)$$
			throughout $\mathbb{CP}^2\setminus \{\mathrm{pt}\}$.
			
			\item On the set $\{ x \in \mathbb{CP}^2\setminus \{\mathrm{pt}\} : \eta_{\Gamma_0}(x)  =  1 \}$,
			$$\acute G_{\mathbf{p}}(t) = G_{\mathbf{p}}(t) 	\qquad
			\text{for all } t \in [t_0, T( \mathbf{p})).$$
			
			\item On the set $\{ x \in \mathbb{CP}^2\setminus \{\mathrm{pt}\} : \eta_{\Gamma_0}(x) =  0\}$,
			$$\acute G_{\mathbf{p}}(t) = (1-t) \phi_{t}^* g\qquad
			\text{for all } t \in [t_0, T(\mathbf{p})).$$
			
			\item $\acute G_{\mathbf{p}}(t)$ does \emph{not} solve Ricci flow.
			However, on open subsets of $\eta_{\Gamma_0}^{-1} ( \{ 1 \} )$ and $\eta_{\Gamma_0}^{-1} ( \{ 0 \} )$, $ \acute G_{\mathbf{p}}(t)$ \emph{does} solve Ricci flow since it's equal to the Ricci flow solution $G_{\mathbf{p}}(t)$ or $(1-t)\phi_t^* g$.
			
			In general,
			\begin{gather*}
				\partial_t \acute G_{\mathbf{p}} 
				= -2 \eta_{\Gamma_0} Ric( G_{\mathbf{p}}(t)) - 2 (1-\eta_{\Gamma_0}) Ric (\phi_t^* g) ,
			\end{gather*}
			which is supported on the closure of $\{ x \in \mathbb{CP}^2\setminus \{\mathrm{pt}\} : 0 < \eta_{\Gamma_0}(x) < 1 \}$.
		\end{enumerate}
	\end{lemma}
	
	\subsubsection{Set up to analyse the convergence to FIK} \label{h}
	
	The time-dependent diffeomorphisms appearing in the Cheeger-Gromov convergence of a blow up sequence coming from the finite time singularity of Theorem \ref{fullthm2} to FIK will be the harmonic map heat flow of the following definition.

	\begin{definition}[\cite{Stol1}, Definition 4.10]\label{HMHF1}
		For $t \in (-\infty, 1)$, let $ \phi( \cdot, t) = \phi_t : \mathbb{CP}^2\setminus \{\mathrm{pt}\} \to \mathbb{CP}^2\setminus \{\mathrm{pt}\}$ denote the soliton diffeomorphisms given as the solution to
		\begin{equation*} 
			\partial_t \phi_t  
			= \frac{1}{1 - t}\overline{\nabla} f \circ \phi_t 
			\qquad \text{ with initial condition } \phi_{0} = Id_{\mathbb{CP}^2\setminus \{\mathrm{pt}\}}.
		\end{equation*}
		Define the one-parameter families of functions
		\begin{equation*} 
			\Phi_t, \tilde	{\Phi}_t: \mathbb{CP}^2\setminus \{\mathrm{pt}\} \to \mathbb{CP}^2\setminus \{\mathrm{pt}\} \qquad \text{ by } \Phi_t = \phi_t \circ \tilde{\Phi}_t
		\end{equation*}
		where $\tilde{\Phi}_t : \mathbb{CP}^2\setminus \{\mathrm{pt}\} \to \mathbb{CP}^2\setminus \{\mathrm{pt}\}$ solves the time-dependent harmonic map heat flow
		\begin{equation}\label{HMHF}
			\partial_t \tilde{\Phi}_t = \Delta_{\acute{G}_{\mathbf{p}}(t), (1-t) \phi_t^* \overline{g} } \tilde{\Phi}_t
			\qquad \text{ with initial condition }
			\tilde{\Phi}_{t_0} = Id_{\mathbb{CP}^2\setminus \{\mathrm{pt}\}}.
		\end{equation}
		Let $T_\Phi( \mathbf{p}) \in [t_0,  T( \mathbf{p})]$ denote the maximal time for which $\tilde{\Phi}_t$, or equivalently $\Phi_t$, exists, is unique, and is a diffeomorphism.
	\end{definition}
	
	\begin{remark}
		The harmonic map heat flow preserves the $U(2)$-symmetry since both $\acute{G}_{\mathbf{p}}(t), (1-t) \phi_t^* g$ are $U(2)$-symmetric.
	\end{remark} 
	
	For all $t\in[t_0, T_{\Phi}(\mathbf{p}))$, define 
	$$g(t)=g_{\mathbf{p}}(t; t_0)=\frac{1}{1-t}(\Phi_{t}^{-1})^* G'_{\mathbf{p}}(t).$$
	\begin{corollary}[\cite{Stol1}, Corollary 4.15] \label{equationforh}
		For all $ t \in ( t_0, T_{\Phi} ( \mathbf{p}) )$, 
		the tensor
		$$h=h_{\mathbf{p}} (t;\Gamma_0,\gamma_0,t_0) \coloneqq g(t) - \overline{g}$$
		satisfies an evolution equation given in local coordinates by
		\begin{gather} \label{equationforh2}
			\begin{aligned}
				( 1 - t) \partial_t h_{ij}	
				= \Delta_{\overline{g},f}h + 2 \overline{Rm}(h)+\mathcal{O}(|h|^2)+\left(\Phi_t^{-1} \right)^* \left\{ \partial_t \acute G_{\mathbf{p}} + 2 Rc[ \acute G_{\mathbf{p}} ] \right\}.	
			\end{aligned} 
		\end{gather}
		where $\mathcal{O}(|h|^2)$ denotes the non-linear terms coming from the linearisation of the evolution equation of $h$.
	\end{corollary}

	\subsubsection{The box}\label{box}
	Theorem \ref{U(2)Spectrum} allows the following decomposition into $U(2)$-symmetric 2-tensors. Recall that we chose $\lambda_{*}\in (-\infty,0)\setminus \{\lambda_{j}\}$ and $K\in \mathbb{N}$ such that $\lambda_{K}>\lambda_{*}>\lambda_{K+1}$.
	\begin{definition} 
		For general $h \in L^2_f(\mathbb{CP}^2\setminus \{\mathrm{pt}\})^{U(2)}$, define the unstable and stable projections by
		\begin{align*} 
			&h_u \coloneqq  \sum_{j=1}^K ( h, h_j)_{L^2_f} h_j,	\text{ and }\\
			&h_s \coloneqq  \sum_{j=K+1}^\infty ( h, h_j)_{L^2_f} h_j = h - h_u.	\\
		\end{align*}
	\end{definition}

	\begin{definition} \label{boxdef}
		For constants $ \mu_u, \mu_s, \epsilon_0, \epsilon_1, \epsilon_2 \in (0,1)$, and an interval $I \subset [0,1)$, 
		define the box
		$$\mathcal{B} = \mathcal{B}[\lambda_* ,\mu_u,\mu_s, \epsilon_0,\epsilon_1, \epsilon_2, I]$$
		to be the set of smooth $U(2)$-symmetric sections $h$ of $ Sym^2 T^*(\mathbb{CP}^2\setminus \{\mathrm{pt}\}) \times I$ 
		such that:
		\begin{enumerate}[i)]
			\item $h(\cdot, t)$ is compactly supported in $\mathbb{CP}^2\setminus \{\mathrm{pt}\}$ for all $t \in I$,
			
			\item ($L^2_f$ estimates)
			\begin{align*} 
				\left \lVert h_u( \cdot, t) \right\rVert_{L^2_f} &\leq \mu_u (1-t_0)^{-\lambda^*}		
				\text{ for all } t \in I, \\
				\left \lVert h_s( \cdot, t) \right\rVert_{L^2_f} &\leq \mu_s (1-t_0)^{-\lambda^*}		
				\text{ for all } t \in I, 
			\end{align*} 
			and
			\item ($C^2$ estimates)
			\begin{align*}
				&\lvert h \rvert_{\overline{g}} 
				\leq \epsilon_0	
				\text{ for all } t \in I, \\
				&\lvert \overline{\nabla} h \rvert_{\overline{g}} \leq \epsilon_1
				\text{ for all } t\in I, \text{ and} \\
				&\lvert \overline{\nabla}^2 h \rvert_{\overline{g}} \leq \epsilon_2
				\text{ for all } t \in I. 
			\end{align*}
		\end{enumerate}
	\end{definition}

	\begin{remark}
		The linearisation of equation (\ref{equationforh2}) for $h$ is
		$$(1-t)\partial_t h=\Delta_f h+2Rm(h).$$
		Thus, if we consider Ricci flow with initial condition
		$$\acute G_{\mathbf p}(t_0) 
		= (1 - t_0) \phi_{t_0}^* g
		+  ( 1 - t_0)  \phi_{t_0}^* \left( \eta_{\gamma_0}    \sum_{j = 1}^{K } p_j  h_{j}  \right),$$
		meaning $h= \eta_{\gamma_{0}} \sum_{i=1}^{K} p_j h_j$, then we expect $h$ to evolve as
		\begin{equation}\label{approxh}
			h(\tau)\approx \sum_{i=1}^{K}p_j (1-t)^{-\lambda_{j}}h_j,
		\end{equation}
		in the region where $\eta_{\gamma_{0}}=1$.
		By Remark \ref{>0}, there is some $1\leq j<K$ with $\lambda_{j}>0$. So we do not expect $h(t)\not\rightarrow 0$ as $t \rightarrow 1$. The Wa{\.z}ewski box argument will show however that there is some $\mathbf{p}$ such that $h(t)\rightarrow 0$ as $t \rightarrow 1$. 
		With the approximation to the linearised flow (\ref{approxh}) in mind, an inequality of the form 
		$$\left \| h_s( \cdot, t) \right\|_{L^2_f} \le \mu_s (1-t)^{-\lambda_*}	\qquad	
		\text{ for all } t\in I,$$
		is expected. 		
		However, an inequality of the form 
		$$\left \| h_u( \cdot, t) \right\|_{L^2_f} \le \mu_u (1-t)^{-\lambda_*}	\qquad	
		\text{ for all } t \in I,$$
		is not expected. Therefore we expect $h$ to leave the box  
		$$\mathcal{B}=\mathcal{B}[ \lambda_* , \mu_u, \mu_s, \epsilon_0, \epsilon_1, \epsilon_2 , [t_0, t_1) ]$$
		in general. The Wa{\.z}ewski box argument will show that all the $h_{\mathbf{p}}(t)$ cannot leave $\mathcal{B}$, provided $\gamma_{0}, \overline{p}, \epsilon_{0}, \epsilon_{1}, \epsilon_{2}, \mu_u, \mu_s, 1-t_0$ are chosen suitably large or small. We will see that this implies there exists a $\mathbf{p^*}$ with $|\mathbf{p}^*| \le \overline{p} (1 - t_0)^{| \lambda_*|}$ such that $h_{\mathbf{p}^*} \in \mathcal{B}[ \lambda_* , \mu_u, \mu_s, \epsilon_0, \epsilon_1, \epsilon_2 , [t_0, 1 )]$.
		Briefly, the reason for this is because if 
		all the $h_\mathbf{p}(t)$ did exit the box, then for $0 < \mu_u , \mu_s\le  2$, we can define 
		$$t_{\mu_u, \mu_s}^* ( \mathbf p ),$$
		to be the infimum of the times $t$ such that $h_\mathbf{p}(t)$  is not in the box. Then for the closed ball $\overline{B}_{\overline{p}} e^{\lambda_* \tau_0}  \subset \mathbb{R}^K$ of radius $\overline{p} e^{\lambda_* \tau_0}$ centred at the origin,
		define the map
		\begin{gather*}
			\mathcal{F} : \overline{B}_{\overline{p} e^{\lambda_* \tau_0} } \to \mathbb{R}^K,\\
			\mathcal{F}(\mathbf p ) 
			= 
			\left( \big( h_{\mathbf p }(t_{\mu_u, \mu_s}^*(\mathbf p) )  , h_1 \big)_{L^2_f(M)}, 
			\dots ,
			\big( h_{\mathbf p }(t_{\mu_u, \mu_s}^*( \mathbf p) )  , h_{K} \big)_{L^2_f(M)} \right).
		\end{gather*}
		
		It will be the case that if we choose  $\overline{p}, \gamma_0,  \epsilon_0, \epsilon_1, \epsilon_2, \tau_0, \mu_u, \mu_s$ to be suitably large or small, the map
		$$\mathcal{F} : \overline{B_{\overline{p} e^{\lambda_* \tau_0} } } \to \mathbb{R}^K \setminus \{ 0 \}$$
		will be a continuous function whose restriction to $$\overline{A} = \{ \mathbf p \in \mathbb{R}^K : 2 \mu_u e^{\lambda_* \tau_0} \le \lvert \mathbf p \rvert \le \overline{p} e^{\lambda_* \tau_0} \}$$
		is homotopic to the identity $\text{Id} : \overline{A} \to \overline{A} \subset \mathbb{R}^K \setminus \{ 0 \}$.
		No such a map exists, yielding a contradiction.
		The purpose of the $C^2$ estimates is to ensure that $h$ does not exit $\mathcal{B}$ due to the failure of the harmonic map heat flow to exist, and to the control the non-linear terms in the equation for $h$.
	\end{remark}

	\subsection{Pseudolocality results}
	As mentioned earlier, the proof of the existence of a Ricci flow in $G_{AF, \mathbb{CP}^2-\{\text{pt}\}}$ which encounters a finite time singularity modelled on FIK differs from the proof of \cite[Theorem 1.1]{Stol1} (with FIK as the asymptotically conical shrinker) only in that we now desire a Ricci flow in $G_{AF, \mathbb{CP}^2-\{\text{pt}\}}$, and so use a $U(2)$-symmetric basis of eigentensors, and are glueing FIK into a non-compact manifold. The only place where compactness is used in \cite{Stol1} is to prove the curvature bound \cite[Proposition 5.2]{Stol1}, which uses a pseudolocality result for Ricci flow on compact manifolds. In this subsection we record a pseudolocality result for Ricci flow on non-compact manifolds with curvature decaying at infinity and then prove the analogue of the curvature bound \cite[Proposition 5.2]{Stol1}. Once this is done, the part of the proof of Theorem \ref{fullthm2} which uses the box argument follows without any additional work.

	We denote by $\omega$ the volume of ball $B^4$ with unit radius in $\mathbb{R}^4$.
	\begin{theorem}[Peng \cite{Peng}]  \label{pseudolocalitynoncompact}
		There exist $\epsilon, \delta > 0$  with the following property:
		If $(M^4, g(t) )$ is a smooth complete Ricci flow on a manifold $M$ with bounded curvature on each time slice and defined for $t \in [t_0, T)$ that, at initial time $t_0$, satisfies
		$$\lvert Rm\rvert_{g(t_0)}(x) \leq \frac{1}{r_0^2} \qquad \text{ for all  } x \in B_{g(t_0)}(x_0, r_0)$$
		and
		$$\mathrm{Vol}_{g(t_0) } B_{g(t_0)}(x_0, r_0) \geq (1-\delta) \omega r_0^4 \qquad \text{ for all } x \in B_{g(t_0)} (x_0,r_0)$$
		for some $x_0\in M$ and $r_0>0$,
		then
		$$\lvert Rm\rvert_{g(t)}(x,t) \leq \frac{1}{\epsilon^2 r_0^2} \qquad \text{for all } x \in B_{g(t)} (x_0, \epsilon r_0), t \in [t_0, \min \{ T , t_0 +  (\epsilon r_0)^2\} ).$$
	\end{theorem}
	
	The following result and proof is an adaptation of \cite[Proposition 5.2]{Stol1}. Recall that $T(\mathbf{p} ) = T(\mathbf{p}; \Gamma_0, \gamma_0, t_0)$ denotes the minimum of $1$ and the maximal existence time of $G_{\mathbf{p}}( t ; \Gamma_0, \gamma_0 , t_0)$
	
	\begin{proposition}\label{pseudo5}
		Let $c>0$. There exists $\mathcal{K}_0=\mathcal{K}_0$ such that if $0<1-t_0 \ll 1$ is small enough depending on $\Gamma_{0}$, then for all $\lvert \mathbf{p} \rvert \leq \overline{p}e^{\lambda_{*} \tau_{0}}$,
		$$\lvert Rm(G_{\mathbf{p}}(t))\rvert_{G_{\mathbf{p}}(t)} (x,t) \leq \mathcal{K}_0,$$
		for all $(x,t)\in \{f\geq \frac{1}{2}\Gamma_{0}+c\} \times [t_0, T(\mathbf{p})].$
	\end{proposition}
	
	\begin{proof}
		Let $\epsilon, \delta  > 0$ be as in Theorem \ref{pseudolocalitynoncompact}. By Remark \ref{remark1}, throughout $\{ f > f_{\mathrm{Bolt}}+\Gamma_0 / 2\}$, we have
		$G_{\mathbf{p} } (t_0) = G_{\mathbf{0}} (t_0)$ for all $\mathbf{p}$. Therefore, for some $C = C(\Gamma_0)$, 
		\begin{equation}\label{Cbound}
			\lvert Rm(G_{\mathbf{p}}(t_0)) \rvert_{G_{\mathbf{p}}(t_0)} \leq C
		\end{equation}
		by Definition \ref{G_0} provided $0 < 1 - t_0 \ll 1$ is sufficiently small.
		Since $G_{\mathbf{0}}(t_0)=g_{\mathrm{Bolt}}$ on $\{f>\tilde{C}\}$ for some $\tilde{C}>\frac{\Gamma_0}{2}+c>0$, we can find $0 < r_0 \ll 1$ sufficiently small such that  $\text{Vol} \,B_{ g_{\mathrm{Bolt}}} (x, r_0) \ge ( 1 - \delta) \omega r_0^4
		\text{ for all } x \in \{f>\tilde{C}\}$. 
		Furthermore, the restriction of the metrics $G_{\mathbf{p} } (t_0)$ to $\{ f >\Gamma_0 / 2\}$  stay within an arbitrarily small $C^0$-neighbourhood for all $0 < 1 - t_0 \ll 1$ sufficiently small. We claim that, by making $0 < r_0 \ll 1$ sufficiently small,  the following holds,
		\begin{equation}\label{vol}
			\text{Vol} (B_{ G_{\mathbf{p}}(t_0)}) (x, r_0) \geq ( 1 - \delta) \omega r_0^4,
		\end{equation}
		for all $x \in \{\frac{\Gamma_0}{2}+c\leq f\leq \tilde{C}\}$. As argued in \cite[Proposition 5.2]{Stol1}, if this is not the case then there are sequences $x_i \in \{\frac{\Gamma_0}{2}+c\leq f\leq \tilde{C}\}, t_i \rightarrow 1, r_i \rightarrow 0$ such that 
		\begin{equation*}
			\text{Vol} (B_{ G_{\mathbf{p}}(t_i)}) (x_i, r_i) < ( 1 - \delta) \omega r_i^4.
		\end{equation*}
		The curvature estimate (\ref{Cbound}) and the lower volume bound of Proposition \ref{G_0} implies that the blow-up sequence $r_i^{-2}G_{\mathbf{p}}(t_i)$ based at $x_i$ has a subsequence which converges in the pointed $C^{1, \alpha}$-topology to the Euclidean metric on $\mathbb{R}^4$. This contradicts the choice of sequences $x_i, t_i, r-i$. Thus, (\ref{vol})  holds for sufficiently small $r_0$.
		
		Therefore, there exists $0 < r_0 \ll 1$ sufficiently small 
		depending on $\Gamma_0$
		so that, for all $0 < 1 - t_0 \ll 1$ sufficiently small,
		\begin{gather*}
			C \leq \frac{1}{r_0^2} ,\\
			B_{G_{\mathbf{p}} (t_0)}(x, r_0) \subset \{f> \Gamma_0/2\}
			\qquad \text{for all } x \in \{f\geq \frac{1}{2}\Gamma_{0}+c\}, \text{ and}\\
			\text{Vol} (B_{G_{\mathbf{p}} (t_0)} (x, r_0)\geq ( 1-\delta)\omega r_0^4
			\qquad \text{for all } x \in \{f\geq \frac{1}{2}\Gamma_{0}+c\}.
		\end{gather*}
		If $0<1-t_0\ll 1$ is also small enough so that
		$$T(\mathbf{p})-t_0 \leq 1-t_0 \leq \epsilon^2 r_0^2,$$ then 
		$$\lvert Rm(G_{\mathbf{p}}(t))\rvert_{G_{\mathbf{p}}(t)} (x,t)  \leq \frac{1}{ \epsilon^2 r_0^2 } \qquad \text{for all } (x,t) \in \{f\geq \frac{1}{2}\Gamma_{0}+c\} \times [t_0, T(\mathbf{p})).$$
		Taking $\mathcal{K}_0 = ( \epsilon r_0)^{-2}$ completes the proof.
	\end{proof}

	\section{Proof of Theorem \ref{fullthm2}}\label{final}
	
	We briefly recall the set up of the box argument of \cite{Stol1} when the asymptotically conical shrinker is FIK so as to make the adaptation to our setting as transparent as possible.

	As is the case throughout this paper, $(\mathbb{CP}^2\setminus \{\mathrm{pt}\},\overline{g})$ is the shrinking soliton FIK. Stolarski constructs metrics
	$G_{\mathbf{0}}(t_0)$ on a compact manifold satisfying properties (1)-(4) (apart from the claim in (3) about Taub-Bolt) of Proposition \ref{G_0}. We also note that the dependence of the curvature (4) on $\Gamma_{0}$ is not actually used in the proof of \cite[Theorem 1.1]{Stol1}).  Fix $\lambda_{*}<0$ and find $K$ such that $\lambda_{K}<\lambda_{*}<\lambda_{K+1}$. In the notation of Stolarski, for each $\mathbf{p}$, define the metic
	$$G_{\mathbf{p}}(t_0)= G_{\mathbf 0 }(t_0 ) 	
	+ ( 1 - t_0) \phi_{t_0}^* \left( \eta_{\gamma_0}     \sum_{j = 1}^{K } p_j  h_{j}   \right),$$
	where $\eta_{\gamma_{0}}$ is the bump function introduced earlier, and the $h_j$ are the not necessarily $U(2)$-symmetric basis of Theorem \ref{Spectrum}. 
	Evolve $G_{\mathbf{p}}(t_0)$ by the Ricci flow
	$$\partial_t G_{\mathbf{p}} (t) = -2 Ric [ G_{\mathbf{p}} (t) ].$$
	Then  Stolarski proves that $G_{\mathbf{p}}(t)$ has curvature bounds 
	$$\lvert Rm(G_{\mathbf{p}}(t))\rvert_{G_{\mathbf{p}}(t)} (x,t) \leq \mathcal{K}_0,$$
	for all $(x,t)\in \{ \frac{4}{6}\Gamma_{0}< f< \frac{5}{6}\Gamma_{0}\} \times [t_0, T(\mathbf{p})],$ and where $\mathcal{K}_0$ is a constant.
	Define a new metric on $M$: truncate $G_{\mathbf{p}}(t)$ and then extend to a metric on $M$ by
	
	$$\acute G_{\mathbf{p}}(t) 
	\coloneqq \eta_{\Gamma_0} G_{\mathbf{p}}(t) + (1- \eta_{\Gamma_0}) (1 - t) \phi_t^* \overline{g},$$
	where $\eta_{\Gamma_0}$ is as defined earlier.	
	With $\Phi_{t}$ defined as in Definition \ref{HMHF1}, define $$g(t)=g_{\mathbf{p}}(t; t_0)=\frac{1}{1-t}(\Phi_{t}^{-1})^* G'_{\mathbf{p}}(t),$$
	and
	$$h_{\mathbf{p}} (t;\Gamma_0,\gamma_0,t_0) \coloneqq g(t) - \overline{g}.$$
	The box
	$$\mathcal{B} = \mathcal{B}[\lambda_* ,\mu_u,\mu_s, \epsilon_0,\epsilon_1, \epsilon_2, I],$$
	is then defined exactly as we have done in Definition \ref{boxdef}.
	The proof of \cite[Theorem 1.1]{Stol1} only uses that the objects are defined as above with $G_{\mathbf{p}}(t)$ having the curvature bounds
	$$\lvert Rm(G_{\mathbf{p}}(t))\rvert_{G_{\mathbf{p}}(t)} (x,t) \leq \mathcal{K}_0=\mathcal{K}_0(c),$$
	for all $(x,t)\in \{f\geq \frac{1}{2}\Gamma_{0}+c\} \times [t_0, T(\mathbf{p})],$	
	(i.e.~ the only properties of $G_{\mathbf{p}}(t_0)$ used in the proof that relates to the Riemannain manifold we glued FIK into is the curvature bound), which are proven in our case via Proposition \ref{pseudo5}. The estimates of the eigentensors $h_j$ of \cite{Stol1} clearly still hold for our $U(2)$-symmetric basis $h_j$ of Theorem \ref{U(2)Spectrum}.
	
	As a result, It will be the case that for our parameters, if we choose  $\overline{p}, \gamma_0,  \epsilon_0, \epsilon_1, \epsilon_2, \tau_0, \mu_u, \mu_s$ to be suitably large or small, the map $\mathcal{F}$ defined informally above and defined precisely in of \cite[Page 72]{Stol1} is continuous, maps into 
	$ \mathbb{R}^K \setminus \{0\}$.
	and whose restriction to $$\overline{A} = \{ \mathbf p \in \mathbb{R}^K : 2 \mu_u e^{\lambda_* \tau_0} \le \lvert \mathbf p \rvert \le \overline{p} e^{\lambda_* \tau_0} \}$$
	is homotopic to the identity $\text{Id} : \overline{A} \to \overline{A} \subset \mathbb{R}^K \setminus \{ 0 \}$.
	No such a map exists, yielding a contradiction. 
	As argued in the proof of \cite[Theorem 1.1]{Stol1}, this implies the existence of a $\mathbf{p}$ such that $G_{\mathbf{p}}(t)$ encounters a finite time singularity modelled on FIK follows.

	\begin{proof}[Proof of Theorem \ref{fullthm2}]
		Sticking with the notation of \cite{Stol1}, we prove Theorem \ref{fullthm2} so that we encounter a singularity at time $1$. By translating time we can yield a Ricci flow with a singularity at time $T$. Also, although our initial metric for the pre-surgery flow  in this proof will be equal to a multiple of $g_{\mathrm{Bolt}}$ outside some compact set, by rescaling or parametrising we can clearly obtain a Ricci flow with initial metric equal to $g_{\mathrm{Bolt}}$ outside some compact set.

		By the Wa{\.z}ewski box argument deployed in the proof of \cite[Theorem 1.1]{Stol1}, provided $\gamma_{0}, \overline{p}, \epsilon_{0}, \epsilon_{1}, \epsilon_{2}, \mu_u, \mu_s, 1-t_0$ are chosen suitably large or small, there exists a $\mathbf{p}^*$ such that
		\begin{enumerate}[i)]
			\item 	$\lvert \mathbf{p^*} \rvert \le \overline{p} (1-t_0)^{-\lambda^*}$,
			\item 	$T( \mathbf{p^*} ) \in[ t_1, 1]$, and
			\item There exists a smooth function $\tilde{\Phi} : \mathbb{CP}^2\setminus \{\mathrm{pt}\} \times [t_0, 1) \to \mathbb{CP}^2\setminus \{\mathrm{pt}\}$ 
			solving (\ref{HMHF}) 
			\begin{equation*}
				\partial_t \tilde{\Phi} = \Delta_{\acute G_{\mathbf{p} }(t), ( 1- t) \phi_t^* \overline{g} } \tilde{\Phi}
				\text{ on } \mathbb{CP}^2\setminus \{\mathrm{pt}\} \times (t_0, 1)
				\qquad \text{with initial condition }
				\tilde{\Phi}( \cdot, t_0) = Id_{\mathbb{CP}^2\setminus \{\mathrm{pt}\}}
			\end{equation*}
			where
			$\tilde{\Phi}_t = \tilde{\Phi}( \cdot, t) : \mathbb{CP}^2\setminus \{\mathrm{pt}\} \to \mathbb{CP}^2\setminus \{\mathrm{pt}\}$ is a diffeomorphism for all $t \in [t_0,1)$ and,
			$$h_{\mathbf{p^*} }(t) \doteqdot \frac{1}{1-t} ( (\phi_{t}\circ \tilde{\Phi}_t)^{-1} )^* \acute G_{\mathbf{p^*}}(t) - \overline{g} 
			\in \mathcal{B}[ \lambda_* ,  \mu_u, \mu_s, \epsilon_0, \epsilon_1, \epsilon_2 , [t_0, 1) ].$$
		\end{enumerate}
		
		Moreover, Stolarksi shows that this implies the Ricci flow $G_{\mathbf{p}^*}(t, t_0, \Gamma_{0})$ encounters a local finite time singularity modelled on FIK at $t=1$.
		The harmonic map heat flow $\tilde{\Phi}$ is $U(2)$-symmetric and so $\Phi$ is $U(2)$-symmetric. Since,
		$$G'_{\mathbf{p^*}}(t)=(\tilde{\Phi}_t)^*((1-t)\phi_t^*(\overline{g}+h_{\mathbf{p^*}}(t))),$$
		up to a $U(2)$-symmetric time-dependent diffeomorphism  $G'_{\mathbf{p^*}}(t)$ is 
		$(1-t)\phi_t^*(\overline{g}+h_{\mathbf{p^*}}(t))$. Since the size of the bolt of $(1-t)\phi_t^*(\overline{g}+h_{\mathbf{p^*}}(t))$ goes to zero as $t \rightarrow 1$ and $\tilde{\Phi}_t$ fixes the bolt, the size of the bolt of $G'_{\mathbf{p^*}}(t)$ goes to zero as $t\rightarrow 1$. By Lemma \ref{remark1}, on $\{ x \in \mathbb{CP}^2\setminus \{\mathrm{pt}\} : \eta_{\Gamma_0}(x)  =  1 \}$,
		$\acute G_{\mathbf{p}^*}(t) = G_{\mathbf{p}^*}(t)$,  the size of the bolt of $G_{\mathbf{p}^*}(t)$ goes to zero as $t\rightarrow 0$. In particular, the curvature of the bolt blows up at time $t=1$.

		By Proposition \ref{diagonalprop}, with the correct radial vector field, $G_{\mathbf{p}^*}(t, t_0, \Gamma_{0}) \in G_{AF, \mathbb{CP}^2-\{\text{pt}\}}$ with some mass $m$ (which we can choose to any positive number just by rescaling).  Then we have Ricci flow in $G_{AF, \mathbb{CP}^2-\{\text{pt}\}}$ with the bolt shrinking to zero size in finite time. By Theorem \ref{cto0} the curvature must blow up on $\{s\leq R\}$ or $\{s< R\}$ for some $R\geq 0$ and so we can perform surgery and continue the flow.
		Indeed, assume $(\mathbb{CP}^2-\{\text{pt}\}, G_{\mathbf{p}^*}(t))_{0\leq t <T}$ is a Ricci flow in $G_{AF, \mathbb{CP}^2-\{\text{pt}\}}$ encountering a finite time singularity at time $1$. Stop the flow at a time $T'$ close to $1$ to get a metric:
		\begin{equation*}
			G_{\mathbf{p}^*}(T')=ds^2+ b^{2}(s,T')(\sigma_{1}^2+\sigma_{2}^2)+c^{2}(s,T')\sigma_{3}^2, \qquad s>0.
		\end{equation*}
		for some smooth functions $b,c:(0,\infty) \rightarrow \infty$ with $\lim_{s\rightarrow 0}b(s)>0, \lim_{s\rightarrow 0}b_{s}(s)=0$ and $c$ being extendible to a smooth odd function on $\mathbb{R}$.
		Cut out all points with $s< \bar{R}$ for any $\bar{R}>R$ to leave a metric:
		\begin{equation*} \label{cutmetric}
			G=ds^2+ b^{2}(s,T')(\sigma_{1}^2+\sigma_{2}^2)+c^{2}(s,T')\sigma_{3}^2, \qquad s\geq \bar{R}.
		\end{equation*}
		Surgery would involve finding a metric:
		\begin{equation*}
			\tilde{G}=ds^2+\tilde{b}^{2}(s)(\sigma_{1}^2+\sigma_{2}^2)+\tilde{c}^{2}(s)\sigma_{3}^2, \qquad s>0,
		\end{equation*}
		on $\mathbb{R}^4$, meaning $\tilde{b}, \tilde{c}$ are able to be extended to smooth odd functions on $\mathbb{R}$, with $\tilde{b}(s)=b(s), \tilde{c}(s)=c(s)$ for all $s\geq \bar{R}$. 
		Given smooth functions $b,c:[\bar{R},\infty)\rightarrow \infty$ with $b_{s},c_{s}\geq 0$ and $c\leq b$, it is easy to obtain smooth functions $\tilde{b},\tilde{c}:(0,\infty)\rightarrow \infty$ with $\tilde{b}_{s}, \tilde{c}_{s}\geq 0$ and $\tilde{c}\leq \tilde{b}$ such that $\tilde{b}, \tilde{c}$ are extendible to smooth odd functions on $\mathbb{R}$, with $\tilde{b}(s)=b(s), \tilde{c}(s)=c(s)$ for all $s\geq \bar{R}$. Notice that the curvature decay condition in Definition \ref{GAF} of $G_{AF, \mathbb{R}^4}$ automatically holds for the surgery described in this subsection as it is a condition about $b,c$ near infinity, which already holds for $G_{\mathbf{p}^*}(T')\in G_{AF, \mathbb{CP}^2\setminus\{\text{pt}\}}$. 
		Once we have such a $\tilde{G} \in G_{AF, \mathbb{R}^4}$, Theorem \ref{Fra1} tells us that Ricci flow starting at $\tilde{G}$ will converge to a member of the Taub-NUT family of mass m in infinite time.
	\end{proof}

	\bibliography{refs}
	\bibliographystyle{amsplain}

\end{document}